\documentclass{amsart}

\usepackage{amsmath,amsthm,amssymb}
\usepackage{algorithmic}
\usepackage{algorithm}
\usepackage{setspace}
\usepackage{latexsym}
\usepackage{color}
\usepackage{comment}
\usepackage{graphicx}
\usepackage{cite}

\newtheorem{theorem}{Theorem}[section]
\newtheorem{lemma}[theorem]{Lemma}
 
\newtheorem{proposition}[theorem]{Proposition}
\newtheorem{observation}[theorem]{Observation}
\newtheorem{corollary}[theorem]{Corollary}

\theoremstyle{definition}
\newtheorem{definition}[theorem]{Definition}

\newtheorem{cclaim}[theorem]{Claim}

\newtheorem*{acnt}{Acknowledgments and Notes}

\theoremstyle{remark}
\newtheorem{remark}[theorem]{Remark}

\newcommand{\zero}{forwarding~}

\newcommand{\Vg}{V(G)}

\newcommand{\pardm}[1]{\prec_{#1}}
\newcommand{\yield}{\triangleleft}
\newcommand{\gpart}[1]{\mathcal{P}(#1)}
\newcommand{\pargpart}[2]{\mathcal{P}_{#1}(#2)}
\newcommand{\gsim}[1]{\sim_{#1}}

\newcommand{\upstar}[1]{\mathcal{U}^{*}(#1)}
\newcommand{\parup}[2]{\mathcal{U}_{#1}(#2)}
\newcommand{\parupstar}[2]{\mathcal{U}^{*}_{#1}(#2)}
\newcommand{\vup}[1]{U(#1)}
\newcommand{\vupstar}[1]{U^*(#1)}
\newcommand{\vparup}[2]{U_{#1}(#2)}
\newcommand{\vparupstar}[2]{U^*_{#1}(#2)}

\newcommand{\what}{\hat}

\newcommand{\parcondxp}[5]{\Psi(#1, #2; #3, #4, #5)}

\newcommand{\pathfamily}[6]{\Lambda(#1; #2, #3; #4)}
\renewcommand{\Gamma}{N}

\newcommand{\parNei}[2]{\Gamma_{#1}(#2)}
\newcommand{\comp}[1]{\mathcal{G}(#1)}

\newcommand{\earint}[1]{int(#1)}
\newcommand{\const}[1]{\mathcal{G}^+(#1)}
\newcommand{\inconst}[1]{\mathcal{G^{-}}(#1)}
\newcommand{\cut}[1]{\delta(#1)}
\newcommand{\parcut}[2]{\delta_{#1}(#2)}
\newcommand{\deficiency}[1]{\mathrm{def}(#1)}

\makeatletter
\renewenvironment{proof}[1][\proofname]{\par
  \normalfont
  \topsep6\p@\@plus6\p@ \trivlist
  \item[\hskip\labelsep{\bfseries #1}\@addpunct{\bfseries.}]\ignorespaces
}{%
  \endtrivlist}

\makeatother

\makeatletter
\def\BOXSYMBOL{\RIfM@\bgroup\else$\bgroup\aftergroup$\fi
  \vcenter{\hrule\hbox{\vrule height.85em\kern.6em\vrule}\hrule}\egroup}
\makeatother
\newcommand{\BOX}{%
  \ifmmode\else\leavevmode\unskip\penalty9999\hbox{}\nobreak\hfill\fi
  \quad\hbox{\BOXSYMBOL}}
\renewcommand\qed{\BOX}
\newenvironment{rmenum}{\begin{enumerate}%
\renewcommand{\labelenumi}{\theenumi}%
\renewcommand{\labelenumi}{{\rm \theenumi}}%
\renewcommand{\theenumi}{(\roman{enumi})}%
}{\end{enumerate}}

\numberwithin{equation}{section}

\begin{document}

\title[Canonical Decomposition]{New Canonical Decomposition in Matching Theory}%

\author{Nanao Kita}
\address{National Institute of Informatics\\
2-1-2 Hitotsubashi, Chiyoda-ku, Tokyo, Japan 101-8430}
\email{kita@nii.ac.jp}



\date{\today}

\begin{abstract}
In matching theory, one of the most fundamental and classical branches of combinatorics,  
{\em canonical decompositions} of graphs are powerful and versatile  tools that form the basis of this theory.
However, the abilities of the known canonical decompositions, that is, the  {\em Dulmage-Mendelsohn}, {\em Kotzig-Lov\'asz}, 
and {\em Gallai-Edmonds} decompositions,  are limited  because they are only applicable  to particular classes of graphs, such as  bipartite graphs, 
or they are  too sparse to provide sufficient information. 
To overcome these limitations, 
we  introduce a new canonical decomposition that  is applicable to all graphs 
and  provides much finer information. 
We focus on the notion of {\em factor-components} 
as the fundamental building blocks of a graph; 
through the factor-components, 
our new canonical decomposition states 
how a graph is organized and how it contains all the maximum matchings. 
The main results that constitute our new theory are the following: 
(i)  
a canonical partial order over the set of factor-components, 
which describes 
how a graph is constructed from its factor-components; 
(ii)
a generalization of the Kotzig-Lov\'asz decomposition, 
which shows the inner structure of each factor-component 
in the context of the entire graph; and  
(iii) 
a canonically described interrelationship between (i) and (ii), 
which integrates these two results into a unified theory of a canonical decomposition. 
These results are obtained in a self-contained way, 
and our proof of the generalized Kotzig-Lov\'asz decomposition contains 
a shortened and self-contained proof of the classical counterpart.

\end{abstract}

\maketitle

\section{Introduction} \label{sec:intro}
This paper introduces a new  canonical decomposition in matching theory. 
In this section, we give a brief explanation of our results. 

Matching theory~\cite{lp1986} is one of the most classical and fundamental fields in combinatorics. 
Given a graph, a {\em matching} is a set of edges
 in which no two are adjacent. 
As small matchings such as a singleton  exist trivially by definition, 
 {\em maximum} matchings typically attracts great interest. 
As can be seen from the definition, 
a matching is a basic way to express pairings of elements, 
and therefore has been intensively studied not only in graph theory~\cite{DBLP:books/daglib/0030488}  
but  also in algebra~\cite{duff1986direct, DBLP:journals/combinatorica/SzegedyS06, DBLP:conf/fct/Lovasz79, geelen2000, lp1986}. 

The role of matching theory in combinatorial optimization is especially important. 
In the  decades since 1965, 
the remarkable growth of combinatorial optimization 
 has been driven by {\em polyhedral combinatorics}~\cite{grotschel2012geometric, schrijver2003}, 
which explores  a systematic and unified approach 
to numerous types of combinatorial problems through linear programming theory. 
The maximum matching problem serves as an archetypal prototype in polyhedral combinatorics~\cite{lp1986, schrijver2003}. 
Therefore,  progress in the theory of matchings 
leads to benefits for the entire field of combinatorial optimization. 

{\em Canonical decompositions} are highly versatile tools that form the foundation of matching theory~\cite{lp1986}. 
A type of structure theorems exist that 
define a uniquely determined partition of a graph,  and then 
use this partition to state the matching theoretic properties of the graph.  
A canonical decomposition is a way to understand graphs that is naturally derived from one of these structure theorems.

The known canonical decompositions 
are the following: the {\em Dulmage-Mendelsohn decomposition}~\cite{dm1958,dm1959,dm1963}, 
 {\em Kotzig-Lov\'asz decomposition}~\cite{kotzig1959a, kotzig1959b, kotzig1960, lovasz1972structure}, 
and  {\em Gallai-Edmonds decomposition}~\cite{gallai1964, edmonds1965}. 
The power  of each canonical decomposition 
 originates partly from its uniqueness for a given graph. 
Therefore, in matching theory, 
the adjective ``canonical'' has come to mean being unique for a given graph, 
and being canonical itself has been considered  important.

However, 
we sometimes encounter problems 
that cannot be solved successfully with these canonical decompositions 
because 
they are applicable to only  particular classes of graphs 
or do not provide sufficient information. 
The  Dulmage-Mendelsohn  and  Kotzig-Lov\'asz decompositions 
target  bipartite graphs and  {\em consistently factor-connected graphs}, respectively. 
The  Gallai-Edmonds decomposition, by definition, 
targets all graphs, 
 but tends to be too sparse  and, therefore, some classes of graphs, such as  {\em factorizable graphs}, fall into trivially irreducible cases, which is a limitation that cannot be disregarded. 

To address these limitations, 
in this paper, we establish the {\em basilica decomposition}, 
which is a new canonical decomposition
that is applicable to all graphs  
and provides much finer information than the Gallai-Edmonds decomposition. 
We derive this new canonical decomposition 
using the notion of {\em factor-components}, 
which serve as  the fundamental building blocks that constitute a graph
 when studying matchings.  
The properties of the maximum matchings are captured by describing both 
how an entire given graph is constructed by factor-components 
and  the  inner structure of each factor-component. 
More precisely, 
the main results that constitute the new canonical decomposition 
are the following: 
\begin{rmenum} 
\item \label{item:intro:order} 
The organization of a given graph in terms of its factor-components can be  
understood as a partially ordered structure. 
The set of factor-components forms a poset 
with respect to a certain canonical binary relation, 
which is similar to the Dulmage-Mendelsohn decomposition. 
\item \label{item:intro:sim} 
A generalization of  the Kotzig-Lov\'asz decomposition is provided that targets general graphs, 
which  describes the inner structure of each factor-component in the context of the entire given graph.   
\item \label{item:intro:cor} 
Although   \ref{item:intro:order} and \ref{item:intro:sim} are established independently, 
they have a certain canonical relationship that enables us to understand a graph 
as an architectural building-like structure in which these ideas are unified naturally. 
The integrated notion obtained from this relationship is our  new canonical decomposition.   
\end{rmenum}

Regarding our proofs, 
we obtain this new canonical decomposition 
without using any known results; 
thus, it is purely self-contained. 
Additionally, 
the proof that establishes the generalization of the Kotzig-Lov\'asz decomposition
contains a greatly shortened and purely self-contained proof for the classical Kotzig-Lov\'asz decomposition.

Considering the important role of canonical decompositions, 
we believe that our results will contribute 
to further development in  combinatorics. 
In fact, several consequential results have been already obtained~\cite{DBLP:conf/cocoa/Kita13, kita2012canonical, kita2014alternative, kita2015graph}. 

The remainder of this paper is organized as follows. 
In Section~\ref{sec:def}, we present preliminary definitions and lemmas. 
In Section~\ref{sec:canonical}, we explain more about the technical background  
of the canonical decompositions 
and what we aim to establish in this paper. 
In Section~\ref{sec:props}, we list some elementary well-known lemmas used in later sections, 
 with self-contained proofs.  
The  new results of this paper  appear in  Section~\ref{sec:nonpositive} onward. 
In Section~\ref{sec:nonpositive}, we provide a statement about {\em consistently factor-connected graphs} that is used in later sections. 
The main theorems that establish the basilica decomposition are then presented;  
we present  \ref{item:intro:order},   \ref{item:intro:sim}, and 
\ref{item:intro:cor} in Sections~\ref{sec:order}, \ref{sec:part},  and \ref{sec:cor}, respectively. 

In Section~\ref{sec:pertinentprops}, we present some  properties of the basilica decomposition. 
In Section~\ref{sec:alg}, we propose a polynomial time algorithm for computing the basilica decomposition.  
Finally, in Section~\ref{sec:conclusion},  we conclude this paper.

\section{Definitions}\label{sec:def}  
\subsection{General Statements} 
For standard definitions and notation for sets, graphs, and algorithms, 
we mostly follow  Schrijver~\cite{schrijver2003}.
In the following, we list those that may be non-standard or exceptional. 
We denote the vertex set of a graph $G$ by $V(G)$ 
and the edge set by $E(G)$. 
We treat paths and circuits as graphs; 
that is, a path is a connected graph in which every vertex is of degree two or less 
and at least one vertex is of degree less than two, 
whereas a circuit is a connected graph in which every vertex is of degree two. 
Given a path $P$ and vertices $x,y\in V(P)$, 
$xPy$ denotes the subpath of $P$ whose ends are $x$ and $y$. 
We sometimes regard a graph as its vertex set. 
As usual, a singleton $\{x\}$ is sometimes denoted simply by $x$. 
In the remainder of this section, unless otherwise stated, 
let $G$ be a graph and let $X\subseteq V(G)$. 
\subsection{Operations on Graphs} 

The subgraph of $G$ induced by $X$ is denoted by $G[X]$, 
and $G[\Vg\setminus X]$ is denoted by $G-X$. 

We denote by $G/X$ the contraction of $G$ by $X$. 
That is, $V(G/X) = V(G)\setminus X \cup \{x\}$, where $x\not\in V(G)$, 
and $E(G/X) = E(G) \setminus E(G[X]) \setminus \parcut{G}{X} \cup S$, 
where $S$ is obtained by replacing each edge $uv\in \parcut{G}{X}$ with $u\in X$ and $v\not\in X$ 
by $xv$. 
Let  $\what{G}$ be a supergraph of $G$, 
and let $F\subseteq E(\hat{G})$. 
We denote by $G+F$ and $G-F$ 
the graphs obtained by adding $F$ to $G$ 
and deleting $F$ from $G$ without removing any vertices, respectively.

The union of two subgraphs $G_1$ and $G_2$ of $G$ 
is denoted by $G_1 + G_2$.

For simplicity, 
regarding these operations of creating a new graph  from given graphs,  
we identify the vertices, edges, and subgraphs 
of the newly created graph with 
those of old graphs to which they naturally correspond.

\subsection{Functions on Graphs} 
A {\em neighbor} of $X$ is a vertex in $V(G)\setminus X$ 
that is joined to a vertex in $X$. 
The set of neighbors of $X$ is denoted by $\parNei{G}{X}$. 
Given $Y, Z\subseteq V(G)$, 
 $E_{G}[Y, Z]$ denotes the set of edges joining $Y$ and $Z$, 
and $\delta_{G}(X)$ denotes $E_{G}[X, V(G)\setminus X]$. 
We sometimes denote $E_G[X, Y]$, $\delta_G(X)$, $\parNei{G}{X}$ simply by   
$E[X,Y]$, $\delta(X)$, $\Gamma(X)$, respectively, 
if their subscripts are apparent from the context.

\subsection{Matchings} 
A set of edges is  a {\em matching} 
if any distinct two are disjoint.  
We say that a matching $M$ {\em covers} a vertex $v$ 
if $v$ is adjacent to an edge in $M$, 
otherwise we say that $M$ {\em exposes} $v$.

A {\em maximum matching} is a matching with the greatest cardinality. 
A {\em perfect matching} is a matching that covers all vertices. 
Note that a perfect matching is a maximum matching but the converse does not necessarily hold. 
A graph is {\em factorizable} if it has a perfect matching.  
A {\em near-perfect matching} is a matching that covers all vertices except for one. 
A graph is {\em factor-critical} if, for any vertex,
 there is a near-perfect matching that exposes it. 
 Let $M$ be a matching of a graph $G$. 
We say that $X$ is {\em closed with respect to} $M$ 
if $\parcut{G}{X}\cap M = \emptyset$. 
We denote $M\cap E(G[X])$ by $M_X$.

We say that a path or circuit is $M$-alternating if edges in $M$ and not in $M$ 
appear alternately. 
More precisely, a circuit $C$ is {\em $M$-alternating} if 
$M\cap E(C)$ is a perfect matching of $C$. 
We define three types of $M$-alternating paths.  
Let $P$ be a path with ends $x$ and $y$. 
We say that $P$ is {\em $M$-saturated} or  {\em $M$-exposed} between $x$ and $y$ 
if $M\cap E(P)$  or $E(P)\setminus M$, respectively,  is a perfect matching of $P$. 
We say that $P$ is {\em $M$-forwarding} from $x$ to $y$ 
if $M\cap E(P)$ is a near-perfect matching of $P$ that exposes $y$. 
Accordingly, a path with one vertex is $M$-forwarding. 
That is, $M$-saturated and -exposed paths  have an odd number of edges,  
whose ends are covered and exposed by $M$, respectively. 
In contrast, an $M$-forwarding path from $x$ to $y$ has an even number of edges, 
 in which $x$ is covered by $M$ as long as this path has any edge whereas $y$ is always exposed.

An {\em ear}  relative to $X$ 
is a path with two distinct ends in $X$ such that 
any other vertex is disjoint from $X$,  
or a circuit  such that exactly one vertex is in $X$.  
Let $P$ be an ear relative to $X$. 
Even if $P$ is a circuit, 
the {\em ends} are the vertices in $V(P)\cap X$, 
and the {\em internal} vertices are those in $V(P)\setminus X$. 
Hence, for convenience,  if $x$ is the only end  and $y$ is an internal vertex of $P$, 
 we  denote by $xPy$  one of the paths on $P$ between $x$ and $y$. 
The set of internal vertices of $P$ is denoted by $\earint{P}$. 
If $\earint{P}$ intersects $Y\subseteq V(G)$, 
then we say that $P$ {\em traverses} $Y$. 
We say that $P$ is an $M$-ear if $P\setminus X$ is an $M$-saturated path.

\subsection{Gallai-Edmonds Family} 
Let $G$ be a graph. 
The set of vertices that are exposed by some maximum matchings is denoted by $D(G)$. 
The set $\parNei{G}{D(G)}$ is denoted by $A(G)$, 
and the set $V(G)\setminus D(G) \setminus A(G)$ is denoted by $C(G)$. 
We call  $\{D(G), A(G), C(G)\}$ 
the {\em Gallai-Edmonds family} of $G$, 
because the {\em Gallai-Edmonds decomposition} is 
derived from a structure theorem regarding $D(G)$, $A(G)$, and $C(G)$.  

\subsection{Factor-Connected Components} 
Let $G$ be a  graph. 
An edge $e\in E(G)$   is  \textit{allowed} if 
there is a maximum matching of $G$ containing $e$. 
Let $C_1,\ldots, C_k$ be the connected components of 
the subgraph of $G$ determined by the union of allowed edges. 
We call $G[C_i]$ a {\em factor-connected component} 
or a {\em factor-component} of $G$ for each $i\in \{1,\ldots, k\}$. 
We denote the set of factor-connected components of $G$ by $\mathcal{G}(G)$.

Thus, a  graph is composed of 
its factor-connected components and  the edges 
joining distinct factor-connected components. 
In addition, 
a set of edges is a maximum matching 
if and only if it is a disjoint union of maximum matchings taken from each factor-component. 
Hence, 
we can regard factor-components as the fundamental building blocks 
that determine the matching structure of a graph.

A factor-component is {\em consistent} if it is disjoint from $D(G)$, 
otherwise it is {\em inconsistent}. 
It is also easily observed that 
a factor-component is a factorizable graph if and only if it is consistent.  
Therefore,  given a maximum matching $M$,  
a factor-component $C$ is consistent if and only if 
$M_C$ is a perfect matching of $C$. 
The sets of consistent and inconsistent factor-components of $G$ are denoted by 
$\const{G}$ and $\inconst{G}$, respectively. 
A graph is {\em  factor-connected} if it consists of only one factor-component. 
In particular, this graph is {\em consistently} factor-connected if its only factor-component is consistent.  
Note that any consistent factor-component is a  consistently factor-connected graph.


\section{Canonical Decompositions and  Aim of Our Study} \label{sec:canonical}
\subsection{Known Canonical Decompositions}
We now explain more technical details of the canonical decompositions 
that were omitted from Section~\ref{sec:intro}. 
 The {\em Dulmage-Mendelsohn}~\cite{dm1958, dm1959, dm1963}, 
  {\em Kotzig-Lov\'asz}~\cite{kotzig1959a, kotzig1959b,kotzig1960,lovasz1972structure}, 
 and  {\em Gallai-Edmonds decompositions}~\cite{gallai1964, edmonds1965} are the three known canonical decompositions and  have been extensively applied. 
They are provided by their respective structure theorems, 
which follow a certain common pattern: 
 \begin{itemize} 
 \item First, define a partition of a given graph into substructures, which is described matching theoretically and is,   by definition,  unique to each graph,   
 such as the Gallai-Edmonds family or the set of factor-components. 
 \item 
 Second, provide  statements  about how the entire  graph is structured and  
 the maximum matchings it contains, 
 such as where in the  graph there are allowed or non-allowed edges,  
 or the matching theoretic properties of the substructures determined by the partition. 
 \end{itemize} 
 Because these partitions are determined uniquely for a given graph,  
 canonical decompositions can provide us with information about  
 all maximum matchings, not just those of them that are specified in some way. 
 They therefore exhibit a powerful and versatile nature. 

The traits of the three canonical decompositions are the following.  
 \begin{itemize} 
 \item 
 The {\em Dulmage-Mendelsohn decomposition} states 
 that, for bipartite graphs, the structure of factor-components can be described as a partially ordered set with respect to a certain binary relation. 
 This decomposition provides an efficient solution of a system of linear equations by utilizing the sparsity of matrices~\cite{duff1986direct}.  Additionally, it is the origin of  {\em principal partition theory}~\cite{nakamura1988},  which is a branch of  {\em submodular function theory}~\cite{fujishige2005}. 
 \item 
The {\em Kotzig-Lov\'asz decomposition} 
captures the structure of consistently factor-connected graphs by defining a certain binary relation 
that is proved to be an equivalence relation. 
This decomposition is especially effective in the polyhedral study of matchings. 
From the Kotzig-Lov\'asz decomposition, 
many important results regarding the perfect matching polytopes have been obtained; 
see Lov\'asz and Plummer~\cite{lp1986} or Schrijver~\cite{schrijver2003} for surveys. 
\item 
Among them, the {\em Gallai-Edmonds decomposition} is probably the best known, 
because  it is the essence of characterizing the size of a maximum matching   
and designing algorithms for computing maximum matchings. 
It has contributed to matching theory from many aspects. 
This decomposition  provides properties of graphs based on the Gallai-Edmonds family. 
Some algorithms for computing the maximum matching algorithms are proposed 
using this decomposition~\cite{lp1986, cheriyan1997}. 
It also has  applications in linear algebra~\cite{DBLP:conf/fct/Lovasz79, geelen2000}. 
\end{itemize}

The exact statements of the three canonical decompositions are given in the following. 
The structures of graphs provided by Theorems~\ref{thm:dm}, \ref{thm:canonicalpartition}, and \ref{thm:gallaiedmonds} 
are the Dulmage-Mendelsohn,  Kotzig-Lov\'asz, and  Gallai-Edmonds 
decompositions, respectively.   

\begin{theorem}[Dulmage and Mendelsohn~\cite{dm1958, dm1959, dm1963}]\label{thm:dm} 
Let $G$ be a bipartite  graph with color classes $A$ and $B$, 
and let $\mathcal{G}(G)$ be denoted by  $\{ G_i : i \in I\}$, 
where $I = \{1,\ldots, |\mathcal{G}(G)|\}$.  
Let $A_i = V(G_i)\cap A$ and $B_i := V(G_i)\cap B$ 
for each $i \in I$. 
Then, there exists a partial order $\pardm{A}$ satisfying the following 
for any $i,j\in I$: 
\begin{enumerate} 
\item If $E[A_j, B_i] \neq\emptyset$, then $G_i\pardm{A} G_j$; and, 
\item if $G_i\pardm{A} H \pardm{A} G_j$ yields $G_i=H$ or $G_j = H$, 
then  $E[A_j, B_i] \neq\emptyset$. 
\end{enumerate} 
\end{theorem}

\begin{theorem}[Kotzig~\cite{kotzig1959a, kotzig1959b, kotzig1960}]%
\label{thm:canonicalpartition}
Let $G$ be a consistently factor-connected graph. 
Define a binary relation $\sim$ as follows: 
for $u, v\in V(G)$, 
$u\sim v$  holds if $G-u-v$ is not factorizable. 
Then,  $\sim$ is an equivalence relation on $V(G)$, 
and accordingly, $\gpart{G}$ is a partition of $V(G)$, 
where $\gpart{G} := V(G)/\sim$. 
\end{theorem} 

\begin{theorem}[the Gallai-Edmonds structure theorem; 
Gallai~\cite{gallai1964}, Edmonds~\cite{edmonds1965}]\label{thm:gallaiedmonds}
For any graph $G$, the following hold: 
\begin{rmenum}
\item 
The graph $G[D(G)]$ consists of $|A(G)| + |V(G)|-2\nu(G)$ 
connected components, and each of them are factor-critical, 
whereas each connected component of $G[C(G)]$ is factorizable. 
\item 
Let $M$ be an arbitrary maximum matching of $G$. Then,  
for each connected component $K$ of $G[D(G)]$, the set $M_K$ is a near-perfect matching of $G$; 
each vertex in $A(G)$ is matched to a vertex in $D(G)$, and furthermore, 
if $u$ and $v$ are distinct vertices from $A(G)$, then the vertices to which they are matched belong to distinct connected components of $G[D(G)]$; 
for each connected connected component $L$ of $G[C(G)]$, the set 
$M_L$ is a perfect matching of $L$. 
\item 
All edges in $E[A(G), D(G)]$ are allowed, whereas no edge in $E(G[A(G)])$ or $E[A(G), C(G)]$ is allowed. 

\end{rmenum} 
\end{theorem}

In addition to the statements in these three structure theorems 
 that derive the canonical decompositions,  
 additional fundamental properties are known for each canonical decomposition. 
 These include properties that use the respective canonical decompositions to describe 
 what happens after basic operations that frequently occur in graph theory,   
 such as adding or deleting   vertices and edges. 
These properties accordingly tell us how to make good use of each canonical decomposition. 
 See, for example, Lov\'asz and Plummer~\cite{lp1986} for these properties.

\subsection{Limitations of Classical Canonical Decompositions} 
Although the above canonical decompositions are quite useful, 
we sometimes encounter problems that cannot be solved with any of them, 
because  each of them only targets  a particular class of graphs 
 or they can be too sparse to provide  sufficient information.  

The Dulmage-Mendelsohn  and  Kotzig-Lov\'asz decompositions only target 
 bipartite graphs and consistently factor-connected graphs, respectively. 
The Gallai-Edmonds decomposition,  by definition, targets  any graph $G$,  
however it mainly focuses  on the structure of $G[A(G)\cup D(G)]$ 
and thus provides little information 
about the remainder of the graph, that is, $G[C(G)]$, 
which can be a vast portion.  
In particular, if a given graph $G$ is factorizable, 
then  $D(G) = A(G) = \emptyset$ and $C(G) = V(G)$ hold;  
thus,  the Gallai-Edmonds decomposition claims nothing about $G$. 
This cannot be disregarded because perfect matchings are themselves a notion 
that attracts intense attention.  
Of course, the classical Kotzig-Lov\'asz decomposition is applicable 
to each factor-component in $\comp{G[C(G)]}$,  ignoring the other part; 
however,  the information obtained by this operation 
is meaningless for the entire given graph in most contexts.   

\subsection{Our New Canonical Decomposition} 
In this paper, we present a new canonical decomposition, 
the {\em basilica decomposition}, that overcomes the limitations of the classical decompositions; that is,   this targets all graphs and  simultaneously provides further information that the Gallai-Edmonds decomposition cannot. 
The main concepts and theorems that constitute this new canonical decomposition 
are the following. 
\begin{rmenum} 
\item \label{item:order}
How a graph is organized from its factor-component  
can be described by a partially ordered structure;  
we find 
a canonically defined partial order  between factor-components, 
which is similar to that in the Dulmage-Mendelsohn decomposition (Theorem~\ref{thm:order}).  
\item \label{item:sim}
We obtain a generalization of the Kotzig-Lov\'asz decomposition 
for general graphs (Theorem~\ref{thm:generalizedcanonicalpartition}). 
This generalization considers the entire structure of a given graph 
and provides finer information than  repeated applications of the classical Kotzig-Lov\'asz decomposition. 
\item \label{item:cor} 
There is a relationship between the above two concepts, 
even though they are defined independently (Theorem~\ref{thm:base}). 
This relationship unites the two notions into a canonical decomposition, 
in which we can view a graph as an architectural building-like structure. 
We name this new canonical decomposition  the {\em basilica decomposition}.  
\end{rmenum}

Note  how this new canonical decomposition is obtained. 
All the new statements are provided with self-contained proofs in this paper, 
except for the algorithmic result in Section~\ref{sec:alg} that computes the basilica decomposition in polynomial time.  
Additionally, our results contains a greatly shortened proof of the Kotzig-Lov\'asz decomposition, which is also completely self-contained.

\section{Basic Properties}\label{sec:props} 
\subsection{On Matchings} 
We now present some basic properties of  matchings. 
We will sometimes use these properties  implicitly.  
These are easily  observed  by parity arguments  
or by taking symmetric differences of matchings, 
and readers familiar with matching theory may wish to skip 
this subsection. 

\begin{lemma}\label{lem:cut2forwarding} 
Let $G$ be a graph and $M$ be a matching of $G$. 
Let $X\subseteq V(G)$ be closed with respect to $M$, 
and let $x \in V(G)\setminus X$ and $y \in X$. 
Let $P$ be a path that is $M$-forwarding from $x$ to $y$ or 
$M$-saturated between $x$ and $y$. 
Let $z\in V(P)$ be the first vertex in $X$  that we encounter  if we trace 
$P$ from $x$.  
Then, $xPz$ is an $M$-forwarding path from $x$ to $z$ 
with $V(xPz)\cap X = \{z\}$. 
\end{lemma}

\begin{lemma}\label{lem:allowed}
Let $G$ be a factorizable graph and  $M$ be a perfect matching of $G$, 
and let $xy \in E(G)\setminus M$.  
The following three properties are equivalent: 
\begin{enumerate}
\renewcommand{\labelenumi}{\theenumi}
\renewcommand{\labelenumi}{{\rm \theenumi}}
\renewcommand{\theenumi}{(\roman{enumi})}

\item \label{item:allowed} The edge $xy$ is allowed in $G$. 
\item \label{item:circuit} There is an $M$-alternating circuit $C$ with $xy\in E(C)$. 
\item \label{item:path} There is an $M$-saturated path between $x$ and $y$. 
\end{enumerate}
\end{lemma}

\subsection{On the Gallai-Edmonds Family and Factor-components} 
We now present some  observations about factor-components. 
Lemmas~\ref{lem:da2path} and \ref{lem:a2d} 
are known statements that can be found  
in Edmonds' algorithm for maximum matchings or the Gallai-Edmonds structure theorem.  
These are  easily confirmed. 
\begin{definition} 
Let $G$ be a graph. 
The set of vertices that are exposed by some maximum matchings are denoted by $D(G)$. 
The set $\parNei{G}{D(G)}$ is denoted by $A(G)$, 
and the set $V(G)\setminus D(G) \setminus A(G)$ is denoted by $C(G)$. 
\end{definition}  

\begin{lemma}\label{lem:da2path}
Let $G$ be a graph and $M$ be a maximum matching. 
\begin{rmenum} 
\item \label{item:da2path:d}
A vertex $x$ is in $D(G)$ if and only if 
there exists an $M$-forwarding path from a vertex exposed by $M$ to $x$. 
\item \label{item:da2path:a} 
If a vertex $x$ is in $A(G)$, 
then there exists an $M$-exposed path between $x$ and a vertex exposed by $M$. 
\end{rmenum}
\end{lemma} 

\begin{lemma} \label{lem:a2d}
Let $G$ be a graph. 
For any maximum matching of $G$, 
the vertex to which a vertex in $A(G)$ is matched is in $D(G)$. 
Accordingly, 
no edge in $\cut{C(G)}$ is allowed. 
\end{lemma} 

The next proposition, which characterizes consistent and inconsistent factor-components, 
 follows immediately from Lemma~\ref{lem:a2d}. 
\begin{proposition}\label{prop:fcomp2dac}  
Let $G$ be a graph. 
A factor-component of $G$ is inconsistent 
if and only if it is a factor-component of $G[A(G)\cup D(G)]$. 
A factor-component of $G$ is consistent 
if and only if it is a factor-component of $G[C(G)]$. 
\end{proposition}

\subsection{On Factor-critical Graphs} 
We now present some fundamental properties of factor-critical graphs. 
Some of these are well-known, but we again present their proofs.  
The next one can be easily obtained  by considering symmetric differences of matchings: 

\begin{lemma} \label{lem:path2root} 
Let $M$ be a near-perfect matching of a graph $G$ that exposes $v\in \Vg$. 
Then, $G$ is factor-critical if and only if for any $u\in \Vg$ there exists 
an $M$-\zero path from $u$ to $v$.
\end{lemma}

Lemma~\ref{lem:path2root}  leads to the following three statements: 
\begin{lemma}\label{lem:fc2union}
Let $G$ be a graph and $M$ be a matching of $G$. 
Let $H_1$ and  $H_2$ be factor-critical subgraphs of $G$ such that 
there exists $v\in V(H_1)\cap V(H_2)$ and, for each $i \in \{ 1, 2\}$, 
$M_{H_i}$ is a near-perfect matching of $H_i$ exposing only $v$. 
Then, $H_1 + H_2$ is factor-critical. 
\end{lemma}
\begin{proof}
Obviously, $M_1\cup M_2$ is a near-perfect matching of $H_1 + H_2$, 
exposing only $v$. 
As $H_1$ and $H_2$ are both factor-critical, 
the claim follows from Lemma~\ref{lem:path2root}. 
\qed
\end{proof}

\begin{proposition}[implicitly stated in Lov\'asz~\cite{lovasz1972a}]\label{prop:fc_choice}
Let $G$ be a factor-critical graph, let $v\in V(G)$,
and let $M$ be a near-perfect matching that exposes $v$. 
Then, for any $e\in \cut{v}$, 
there is an $M$-ear relative to $v$ that contains $e$.

\end{proposition}

\begin{proof}
Let $u\in V(G)$ be the  end of the edge $e$ other than $v$. 
From Lemma~\ref{lem:path2root}, 
there is an $M$-forwarding path $P$ from $u$ to $v$. 
Thus, $P+e$ is a desired $M$-ear. 
\qed
\end{proof}

\begin{theorem}[implicitly stated in Lov\'asz~\cite{lovasz1972a}]\label{thm:fc_nice}
Let $G$ be a factor-critical graph. 
For any  factor-critical subgraph $G'$ such that $G-V(G')$ is factorizable, 
the graph $G/G'$ is factor-critical.
\end{theorem}

\begin{proof}
Let $M$ be a perfect matching of $G-V(G')$. 
Note that $M$ is also a near-perfect matching of $G/G'$ 
that exposes the  vertex $g'$ corresponding to $G'$.   
Arbitrarily choose $v\in V(G')$, 
and let $M'$ be a near-perfect matching of  $G'$ that exposes $v$. 
Then, $M'\cup M$ is a near-perfect matching of $G$ that exposes $v$. 
Let $x$ be an arbitrarily chosen vertex in $V(G)\setminus V(G')$. 
From Lemma~\ref{lem:path2root}, there is an $M'\cup M$-forwarding path $P$ from $x$ to $v$. 
Trace $P$ from $x$, and let $y$ be the first encountered vertex in $V(G')$. 
Then, in the graph $G/G'$, the path $xPy$ corresponds to an $M$-forwarding path 
from $x$ to $g'$. 
Hence, from Lemma~\ref{lem:path2root} again, $G/G'$ is factor-critical. 
\qed
\end{proof}

\section{Structure of Alternating Paths in Consistently Factor-connected Graphs} \label{sec:nonpositive}
We now present our new results. 
In this section,  we prove a proposition about consistently factor-connected graphs 
to be used in later sections. 
\begin{proposition}\label{prop:nonpositive}
Let $G$ be a consistently factor-connected graph and $M$ be a perfect matching of $G$.  
Then, for any two vertices $u,v \in V(G)$, 
there is an $M$-saturated path between $u$ and $v$, 
or an $M$-\zero path from $u$ to $v$.
\end{proposition}
\begin{proof}
Let $u\in V(G)$ be an arbitrary vertex. 
Let $U\subseteq V(G)$ be the set of vertices 
that can be reached from $u$ by $M$-saturated or $M$-forwarding paths. 
We obtain this proposition by showing $U = V(G)$.
Suppose, to the contrary,  $U \subsetneq V(G)$.

\begin{cclaim}\label{claim:nonpositive:contained}
Let $v\in U$, and let $P$ be an $M$-saturated path between $u$ and $v$ or an $M$-forwarding path from $u$ to $v$. Then, $V(P)\subseteq U$ holds. 
\end{cclaim} 
\begin{proof} 
Let $w\in V(P)$. 
Then, $uPw$ is an $M$-saturated path between $u$ and $w$ or an $M$-forwarding path from $u$ to $w$. 
Therefore, $w\in U$ holds. 
Hence, we have $V(P)\subseteq U$. 
\qed 
\end{proof}

As $G$ is connected, it has some edges that join $U$ and $V(G)\setminus U$.

\begin{cclaim}\label{claim:nonpositive:nosaturate}  
Let $v\in U \cap \Gamma(V(G)\setminus U)$. 
Then, there is no $M$-saturated path between $u$ and $v$. 
\end{cclaim}
\begin{proof} 
Suppose this claim fails, 
and let $P$ be an $M$-saturated path between $u$ and $v\in U \cap \Gamma(V(G)\setminus U)$. 
From Claim~\ref{claim:nonpositive:contained}, 
$V(P)\subseteq U$ holds. 
Therefore, the vertex $v'$ is in $U$, and by letting  $w \in V(G)\setminus U$ be a vertex with $vw \in E(G)$ we have $vw \not\in M$. 
Hence, $P + vw$ is an $M$-forwarding path from $u$ to $w$, which contradicts $w\not\in U$, 
and this claim is proved. 
\qed
\end{proof}

\begin{cclaim}\label{claim:nonpositive:notinm}
No edge joining $U$ and $V(G)\setminus U$  is in $M$. 
\end{cclaim} 
\begin{proof} 
Let $vw$ be an edge with $v\in U$ and $w\in V(G)\setminus U$. 
From Claims~\ref{claim:nonpositive:contained} and \ref{claim:nonpositive:nosaturate}, 
there is an $M$-forwarding path $P$ from $u$ to $v$ with $V(P)\subseteq U$ .
Hence, if $vw\in M$ then $P+vw$ is an $M$-saturated path between $u$ and $w$, 
and this contradicts $w\not\in U$. 
Therefore, $vw \not \in M$ follows, and this claim is proved. 
\qed
\end{proof}

As $G$ is factor-connected, 
 some edges in $E[U, V(G)\setminus U]$ are allowed. 
Let $e = vw$ be one of these edges. 
From Claim~\ref{claim:nonpositive:notinm}, $e\not\in M$ holds, 
and therefore, from Lemma~\ref{lem:allowed}, 
there is an $M$-saturated path $Q$ between $v$ and $w$. 
Trace $P$ from $u$, and let $x$ be the first vertex we encounter 
that is in $Q$; 
such  $x$ certainly exists under the current hypotheses because $v\in V(P)\cap V(Q)$ holds. 
Note that by this definition of $x$, 
$uPx + xQ\alpha$ forms a path for each $\alpha \in \{v, w\}$. 

\begin{cclaim}\label{claim:nonpositive:upx} 
The path $uPx$ is $M$-forwarding from $u$ to $x$. 
\end{cclaim} 
\begin{proof} 
Suppose this claim fails, that is,  $uPx$ is an $M$-saturated path. 
Then, we have $x' \in V(uPx)$; however, at the same time, 
we have $x' \in V(Q)$, because $x\in V(Q)$ holds and $Q$ is an $M$-saturated path.  
This contradicts the definition of $x$, and this claim is proved. 
\qed
\end{proof} 

Note also that, for $\alpha$, which is equal to either $v$ or $w$, 
$xQ\alpha$ 
 is an $M$-saturated path. 
Hence, from Claim~\ref{claim:nonpositive:upx}, for this $\alpha$, it follows that 
$uPx + xQ\alpha$ is an $M$-saturated path between $u$ and $\alpha$. 
Thus, $w\in U$ holds, which is a contradiction. 
This completes the proof of this proposition.

\qed 
\end{proof}

\section{Partially Ordered Structure}\label{sec:order} 
In this section, we prove that the factor-components of a  graph 
form a partially ordered set with respect to a certain canonical binary relation that we define here.  
As stated before, 
factor-components are the fundamental building blocks of a graph, in that 
a graph consists of its factor-components and the edges between them. 
However, 
how a graph is constructed from factor-components and edges 
is not arbitrary   but follows a certain rule. 
That is, given some factor-connected graphs,   
construct  a new graph by joining them with edges in an arbitrary manner. 
The  factor-components of the resulting graph will not be in general 
equal to the original factor-connected graphs.  
We show that the rule of the factor-components  is an ordered structure.

\begin{definition} 
Given a graph $G$, 
a set $X\subseteq V(G)$ is {\em separating} 
if it is a disjoint union of the vertex sets of some factor-components, 
i.e., if there exist $H_1,\ldots, H_k\in\comp{G}$, where $k\ge 1$, 
such that $X = V(H_1)\dot{\cup}\cdots \dot{\cup} V(H_k)$. 
\end{definition} 
Note that a nonempty set $X$ is separating if and only if
$\cut{X}\cap M = \emptyset$ holds for any maximum matching $M$.    

\begin{definition}
Let $G$ be a  graph, and let $G_1,G_2\in\mathcal{G}(G)$. 
A separating set $X$ is a {\em critical-inducing set for} $G_1$ 
if $V(G_1)\subseteq X$ holds and $G[X]/G_1$ is a factor-critical graph. 
Moreover, 
we say that $X$ is a {\em critical-inducing set for} $G_1$ {\em to} $G_2$ 
if $V(G_1)\cup V(G_2)\subseteq X$ holds and $G[X]/G_1$ is a factor-critical graph. 

We say $G_1\yield G_2$ if 
there is a critical-inducing set for $G_1$ to $G_2$. 
\end{definition}

We show that $\yield$ is a partial order in Theorem~\ref{thm:order}. 
Reflexivity is obvious from the definition, 
hence the following lemmas are provided for transitivity and antisymmetry. 
First of all, observe the following:  
\begin{lemma}\label{lem:order2const}
Let $G$ be a graph. 
If $X$ is a critical-inducing set for a factor-component $H \in\comp{G}$ such that $X \neq V(H)$, 
then $X\setminus V(H)\subseteq C(G)$ holds.  
Consequently, for any maximum matching $M$ of $G$, 
$M_{X\setminus V(H)}$ is a perfect matching of $G[X\setminus V(H)]$. 
Accordingly, 
if $G_1\yield G_2$ holds for 
two distinct factor-components $G_1$ and $G_2$, 
then $G_2$ is  consistent.

\end{lemma}

\begin{proof} 
As $G[X]/H$ is factor-critical, 
 $X\setminus V(H)$ is a separating set such that $G[ X \setminus V(H)]$ has a perfect matching. 
 Therefore, the factor-components that comprise $X\setminus V(H)$ are consistent, 
 which implies from  Proposition~\ref{prop:fcomp2dac} that they are contained in $C(G)$. 
 The remaining claims now follows immediately. 
\qed
\end{proof}

The following three lemmas can be easily confirmed  
by analogy between  factor-critical graphs and  critical-inducing sets. 
The next lemma follows from Lemmas~\ref{lem:path2root} and \ref{lem:cut2forwarding}.  
\begin{lemma}\label{lem:path2base}
Let $G$ be a  graph, $M$ be a maximum matching of $G$, 
and $X\subseteq V(G)$ be a separating set, and let $G_1\in\mathcal{G}(G)$. 
The following three statements are equivalent.  
\begin{rmenum} 
\item \label{item:path2base:main} The set $X$ is a critical-inducing set for $G_1$. 
\item \label{item:path2base:pathcut} For any $x\in X\setminus V(G_1)$, 
there exists $y\in V(G_1)$ such that there is an $M$-forwarding path from $x$ to $y$ 
whose vertices except $y$ are in $X\setminus V(G_1)$. 
\item \label{item:path2base:pathin} For any $x\in X\setminus V(G_1)$, 
there exists $y\in V(G_1)$ such that there is an $M$-forwarding path from $x$ to $y$. 
\end{rmenum}  
\end{lemma}

The next lemma is immediate from Lemma~\ref{lem:fc2union}. 
\begin{lemma}\label{lem:union} 
Let $G$ be a  graph, 
and let $G_1\in \mathcal{G}(G)$. 
If $X_1, X_2 \subseteq V(G)$ are critical-inducing sets for $G_1$, 
then $X_1\cup X_2$ is also a critical-inducing set for $G_1$. 
\end{lemma} 

The next one is easily obtained from Proposition~\ref{prop:fc_choice} and Theorem~\ref{thm:fc_nice}. 

\begin{lemma} \label{lem:inductive-ear} 
Let $G$ be a  graph and $M$ be a maximum matching of $G$, 
and let $G_1\in\mathcal{G}(G)$. 
Let $X$ and $X'$ be critical-inducing sets for $G_1$ 
with $X'\subseteq X$. 
Then, $G[X]/X'$ is factor-critical, 
 and $M_{X\setminus X'}$ is a near-perfect matching of it that exposes only the contracted vertex corresponding to $X'$. 
Moreover, if $X\subsetneq X$  holds, 
then there exists an $M$-ear relative to $X'$ 
whose internal vertices are not empty and are contained in $X\setminus X'$. 
\end{lemma}

Transitivity of $\yield$ now follows rather easily:  

\begin{lemma}\label{lem:transitivity} 
Let $G$ be a  graph 
and $G_1$, $G_2$, $G_3$ be factor-components of $G$. 
If $G_1 \yield G_2$ and $G_2 \yield G_3$ hold, 
then $G_1 \yield G_3$ holds. 
\end{lemma} 
\begin{proof} 
Let $M$ be a maximum matching of $G$.  
Let $X_1$ and $X_2$ be critical-inducing sets for $G_1$ to $G_2$ 
and for $G_2$ to $G_3$, respectively. 
We  prove that $X_1\cup X_2$ is a critical-inducing set for $G_1$ to $G_3$. 
First, $X_1\cup X_2$ is obviously a separating set that contains  $G_1$ and $G_3$. 
Take $x\in X_1\cup X_2$ arbitrarily. 
If $x\in X_1$ holds, then,  from  Lemma~\ref{lem:path2base},  there exists an $M$-forwarding path $P_x$ 
from $x$ to a vertex in $V(G_1)$ with $V(P_x)\subseteq X_1$. 
If $x\in X_2\setminus X_1$ holds, then, from Lemma~\ref{lem:path2base}, 
there is an $M$-forwarding path $Q_x$ from $x$ to a vertex in $G_2$. 
From Lemma~\ref{lem:cut2forwarding},  
there exists $y\in X_1$ such that $xQ_xy$ is an $M$-forwarding path with $V(xQ_xy)\cap X_1 = \{y\}$. 
We  obtain an $M$-forwarding path from $x$ to a vertex in $V(G_1)$, 
that is, $xQ_xy+P_y$. 
Therefore, from Lemma~\ref{lem:path2base}, 
$X_1\cup X_2$ is a critical-inducing set for $G_1$ to $G_3$, 
and the proof is complete. 
\qed
\end{proof}

In the following, 
we  provide definitions and lemmas to prove  antisymmetry of $\yield$. 
\begin{definition}
Let $G$ be a  graph 
and $M$ be a maximum matching of $G$. 
Let $X_0$ be a nonempty proper subset of $V(G)$.   
\begin{rmenum} 
\item Let $X\subseteq V(G)$ be a nonempty set of vetices that is disjoint from $X_0$ and is 
closed with respect to $M$. 
\item Let $P$ be an $M$-ear relative to $X_0$ with $\earint{P}\neq\emptyset$ and $\earint{P}\subseteq X$. 
\end{rmenum} 
For each $x\in X$, define a set of paths $\pathfamily{x}{X}{P}{X_0}{M}{G}$ as follows: 
A path $Q$ is an element of $\pathfamily{x}{X}{P}{X_0}{M}{G}$ 
if it is  $M$-forwarding  from $x$ to a vertex $y\in \earint{P}$  
with $V(Q)\subseteq X$ and $V(Q) \cap V(P) = \{y\}$.  
Additionally, we define a property $\parcondxp{X}{P}{X_0}{M}{G}$ as follows:  
$\parcondxp{X}{P}{X_0}{M}{G}$ is true if $\pathfamily{x}{X}{P}{X_0}{M}{G} \neq \emptyset$ 
for each $x\in X$. 
\end{definition}

Lemmas~\ref{lem:int2root} to \ref{lem:order} in the following present propertites of $\Psi$. 
Among these lemmas, Lemma~\ref{lem:order} is used directly in the proof of Theorem~\ref{thm:order}. 
For two factor-components $G_1$ and $G_2$ with $G_1\yield G_2$, this lemma states that there exist $X$ and $P$ with $\parcondxp{X}{P}{G_1}{M}{G}$ and $V(G_2)\subseteq X$. 
This statement derives that $G_2 \yield G_1$ does not hold and proves antisymmetry of $\yield$. 
Lemmas~\ref{lem:int2root}, \ref{lem:compclosure}, and \ref{lem:extension} are used to prove Lemma~\ref{lem:order}. Lemma~\ref{lem:component2cut} is provided for Lemma~\ref{lem:compclosure}.

Lemma~\ref{lem:int2root} is derived rather easily by considering contatenation of paths.  

\begin{lemma}\label{lem:int2root} 
Let $G$ be a  graph, $M$ be a maximum matching of $G$, 
and $X_0$ be a nonempty proper subset of $V(G)$.  
If $\parcondxp{X}{P}{X_0}{M}{G}$ holds for $X$ and $P$, 
then, for any $x\in X$, $P$ has an end $w$ such that  there exists an $M$-forwarding path $R$ from $x$ to $w$ with $V(R)\setminus \{w\} \subseteq X$.  
\end{lemma} 
\begin{proof} 
Let $x\in X$, and let $Q\in \pathfamily{x}{X}{P}{X_0}{M}{G}$ 
be an $M$-forwarding path from $x$ to a vertex $y \in \earint{P}$. 

Let $w$ be the end of $P$ such that $yPw$ is an $M$-forwarding path from $y$ to $w$. 
Then, $Q + yPw$ is an $M$-forwarding path with the desired property. 
\qed 
\end{proof}

The next lemma is an observation about the intersection of a consistent factor-component 
and a set of vertices closed with respect to a maximum matching.

\begin{lemma} \label{lem:component2cut} 
Let $G$ be a  graph  
and $M$ be a maximum matching of $G$. 
Let $X\subseteq V(G)$ be closed with respect to $M$, 
and let $H\in\const{G}$ be such that $V(H)\cap X \neq \emptyset$. 
Then, for any $x\in V(H)$,  there exist a vertex $y\in X$ 
and an $M$-forwarding path $P$ from $x$ to $y$ with $V(P)\setminus \{y\} \subseteq V(H)\setminus X$.  
\end{lemma}
\begin{proof} 
Take $z\in X\cap V(H)$ arbitrarily. 
From Proposition~\ref{prop:nonpositive}, 
there is a path $Q$ that is $M$-forwarding from $x$ to $z$ 
or $M$-saturated between $x$ and $z$. 
Trace $Q$ from $x$, and let $y$ be the first vertex in $X$ that we encounter. 
Then, from Lemma~\ref{lem:cut2forwarding}, 
$xQy$ is a desired path. 
\qed
\end{proof}

The next lemma is  derived from Lemma~\ref{lem:component2cut} 
and is used to prove Lemma~\ref{lem:order}. 

\begin{lemma}\label{lem:compclosure} 
Let $G$ be a  graph and $M$ be a maximum matching of $G$. 
Let $X_0$ be a nonempty proper subset of $V(G)$ that is separating. 
If a set of vertices $X \subseteq C(G)$ and an $M$-ear $P$ relative to $X_0$ 
satisfy $\parcondxp{X}{P}{X_0}{M}{G}$, 
then 
$X^*$ is a separating set that satisfies $\parcondxp{X^*}{P}{X_0}{M}{G}$, 
where $X^* := X \cup \bigcup \{ V(H): H\in\mathcal{G}(G), V(H)\cap X \neq \emptyset\}$.  
Accordingly, if $X_0 = V(G_1)$ for some $G_1\in\mathcal{G}(G)$, 
then $X^*\cup V(G_1)$ is a critical-inducing set for $G_1$. 
\end{lemma}  
\begin{proof} 
First confirm that $X^*$ is disjoint from $X_0$ and is separating.  
Obviously, 
for each $x\in X$, any path in 
$\pathfamily{x}{X}{P}{X_0}{M}{G}$  is  also a path in $\pathfamily{x}{X^*}{P}{X_0}{M}{G}$, 
and therefore $\pathfamily{x}{X^*}{P}{X_0}{M}{G} \neq \emptyset$. 
Hence, it suffices to prove $\pathfamily{x}{X^*}{P}{X_0}{M}{G} \neq \emptyset$ for each  $x\in V(H)\setminus X$, where $H$ is a factor-component with  $V(H)\cap X \neq \emptyset$. 
As $X\subseteq C(G)$ holds, Proposition~\ref{prop:fcomp2dac} implies that $H$ is consistent. 
From Lemma~\ref{lem:component2cut}, 
there is an $M$-forwarding path $R$ from $x$ to a vertex $y \in X$ with $V(R)\setminus \{y\} \subseteq V(H)\setminus X$. 
For $Q \in \pathfamily{y}{X}{P}{X_0}{M}{G}$,  the concatenation $R + Q$ is a path in $\pathfamily{y}{X^*}{P}{X_0}{M}{G}$. 
Thus, 
we obtain $\parcondxp{X^*}{P}{X_0}{M}{G}$. 
Accordingly, the remaining claim of the lemma also follows 
from Lemmas~\ref{lem:path2base} and \ref{lem:int2root}.  
\qed
\end{proof}

Note also the following observation about $\Psi$, which is used in the proof of Lemma~\ref{lem:order}. 

\begin{lemma} \label{lem:extension} 
Let $G$ be a  graph, $M$ be a maximum matching of $G$, 
and $X_0$ be a nonempty proper subset of $V(G)$.  
\begin{rmenum} 
\item \label{item:extension:ground} If $P_0$ is an $M$-ear relative to $X_0$ with $\earint{P_0} \neq \emptyset$, 
then $\parcondxp{\earint{P_0}}{P_0}{X_0}{M}{G}$ holds. 
\item \label{item:extension:extension} Let $X$ and $P$ be such that  $\parcondxp{X}{P}{X_0}{M}{G}$ holds, 
and let $Q$ be an $M$-ear relative to $X$ with $\earint{Q}\neq\emptyset$ and $\earint{Q}\cap X_0 = \emptyset$. Then, $\parcondxp{X\cup\earint{Q}}{P}{X_0}{M}{G}$ also holds. 
\end{rmenum} 
\end{lemma} 
\begin{proof} 
The property $\parcondxp{\earint{P_0}}{P_0}{X_0}{M}{G}$  holds 
because each vertex  $y\in \earint{P_0}$ forms a trivial $M$-forwarding path 
of $\pathfamily{y}{\earint{P_0}}{P_0}{X_0}{M}{G}$. 
For each $x\in X$, obviously $\pathfamily{x}{X\cup \earint{Q}}{P}{X_0}{M}{G} \neq \emptyset$ holds. 
For each $x\in \earint{Q}$, let $w$ be the end of $Q$ such that $xQw$ is an $M$-forwarding path from $x$ to $w$. For $R\in \pathfamily{w}{X}{P}{X_0}{M}{G}$, $xQw + R$ is a path of $\pathfamily{x}{X\cup\earint{Q}}{P}{X_0}{M}{G}$. 
Thus, $\parcondxp{X\cup\earint{Q}}{P}{X_0}{M}{G}$ is proved.  
\qed
\end{proof}

The next lemma is the key  to Theorem~\ref{thm:order}. 
\begin{lemma}\label{lem:order} 
Let $G$ be a  graph and $M$ be a maximum matching of $G$. 
Let $G_1, G_2\in\mathcal{G}(G)$ be such that $G_1\neq G_2$ and $G_1\yield G_2$ hold. 
Then there exists a set of vertices $X\subseteq V(G)$ and an $M$-ear $P$ relative to $G_1$ such that $V(G_2)\subseteq X$ and $\parcondxp{X}{P}{G_1}{M}{G}$ hold. 
\end{lemma} 
\begin{proof} 
Let $X\subseteq V(G)$ be a critical-inducing set for $G_1$ to $G_2$. 
Define a family $\mathcal{Y}\subseteq 2^{X\setminus V(G_1)}$ as follows: 
A set of vertices $W$ is a member of $\mathcal{Y}$ 
if $W$ is a (inclusion-wise) maximal subset of $X\setminus V(G_1)$ 
that satisfies $\parcondxp{W}{P}{G_1}{M}{G}$ for some $M$-ear $P$ relative to $G_1$.

Let $X' :=  V(G_1) \cup \bigcup_{W \in\mathcal{Y}} W  = \bigcup_{W \in\mathcal{Y}} V(G_1) \cup W$. 
Lemma~\ref{lem:compclosure} implies that, for  each  $W \in \mathcal{Y}$, 
$V(G_1) \cup W$ is a critical-inducing set for $G_1$. 
Accordingly, from Lemma~\ref{lem:union}, $X'$ is also a critical-inducing set for $G_1$.

We prove this lemma by showing 
 $V(G_2)\subseteq X'$. 
Suppose the contrary.  
Then,  $X' \subsetneq X$ holds. 
From Lemma~\ref{lem:inductive-ear}, 
there exists an $M$-ear $Q$ relative to $X'$ 
such that $\earint{Q}\neq \emptyset$ and $\earint{Q}\subseteq X\setminus X'$ hold. 
In the following, note that  
if any $W\subseteq X\setminus V(G_1)$  with  $W \cap \earint{Q} \neq \emptyset$ 
satisfies $\parcondxp{W}{P}{G_1}{M}{G}$ for some $M$-ear $P$, 
then it contradicts the definition of $\mathcal{Y}$ under the current hypothesis.    

Let $\mathcal{Y}^* := \mathcal{Y} \cup \{V(G_1)\}$. 
First consider the case where both ends of $Q$ 
are contained in the same member of $\mathcal{Y}^*$, say,  $W$. 
If $W\in \mathcal{Y}$ holds, then, from Lemma~\ref{lem:extension}\ref{item:extension:extension}, $W \cup \earint{Q}$ is a set of vertices that satisfies  
$\parcondxp{W \cup \earint{Q}}{P}{G_1}{M}{G}$, where $P$ is an $M$-ear with $\parcondxp{W}{P}{G_1}{M}{G}$; 
otherwise, that is, if $W = V(G_1)$, then Lemma~\ref{lem:extension}\ref{item:extension:ground} implies  $\parcondxp{\earint{Q}}{Q}{G_1}{M}{G}$. 
This is a contradiction. 

Next consider the case where 
two ends $u_1$ and $u_2$  of $Q$ are contained in distinct members of $\mathcal{Y}^*$, 
say, $W_1$ and $W_2$, respectively. 
For each $i\in \{1, 2\}$, 
if $W_i$ is a member of $\mathcal{Y}$, 
then let $P_i$ be an $M$-ear relative to $G_1$ 
such that $\parcondxp{W_i}{P_i}{G_1}{M}{G}$ holds;  
according to  Lemma~\ref{lem:int2root}, 
there exists an $M$-forwarding path $R_i$ from $u_i$ 
to an end of $P_i$,  say, $r_i$.  
Otherwise, that is, if $W_i = V(G_1)$, 
 let $R_i$ be the trivial $M$-forwarding path that consists solely of $u_i$, 
and let $r_i := u_i$.

If $(V(R_1)\setminus \{r_1\}) \cap (V(R_2)\setminus \{r_2\}) = \emptyset$, 
then let $\hat{Q} := R_1 + Q + R_2$.   
 From Lemma~\ref{lem:extension}, 
 $\hat{Q}$ is an $M$-ear relative to $G_1$ 
  such that  $\parcondxp{\earint{\hat{Q}}}{\hat{Q}}{G_1}{M}{G}$ holds, 
which is a contradiction. 

Otherwise, that is, if $(V(R_1)\setminus \{u_1\}) \cap (V(R_2)\setminus \{u_2\}) \neq \emptyset$, then 
we have $W_1, W_2 \in \mathcal{Y}$. 
Trace $R_2$ from $u_2$, and let $x$ be the first vertex we encounter that is in $W_1$. 
Then, from Lemma~\ref{lem:cut2forwarding}, 
$u_2R_2x$ is an $M$-forwarding path from $u_2$ to $x$, 
and therefore $Q + u_2R_2x$ is an $M$-ear relative to $W_1$. 
Thus, this case is reduced to the first case, 
and we are again lead to a contradiction. 

Hence, we obtain $X = X'$, and therefore $V(G_2)$ is contained in some member 
of $\mathcal{Y}$. 
This completes the proof. 
\qed
\end{proof}

We can finally prove that $\yield$ is a partial order. 
\begin{theorem}\label{thm:order}
For any  graph $G$, 
the binary relation $\yield$ is a partial order on $\mathcal{G}(G)$. 
\end{theorem}
\begin{proof}
Reflexivity is obvious from the definition.
Transitivity follows from Lemma~\ref{lem:transitivity}.
Hence, we  prove  antisymmetry. 
Let $G_1, G_2\in \mathcal{G}(G)$ be factor-components with $G_1\yield G_2$ and $G_2\yield G_1$. 
Suppose  antisymmetry fails, that is,  $G_1 \neq G_2$ holds. 
Note that, from Lemma~\ref{lem:order2const}, $G_1$ is consistent.   
Let $M$ be a maximum matching of $G$. 
From Lemma~\ref{lem:order}, 
there exist a set of vertices $X\subseteq V(G)$ with $V(G_2)\subseteq X$ 
and an $M$-ear $P_1$ relative to $G_1$ 
that satisfy $\parcondxp{X}{P_1}{G_1}{M}{G}$. 
Let $u_1$ and $v_1$ be the ends of $P_1$.

By Lemma~\ref{lem:path2base}, 
there exists $w\in V(G_2)$ such that   
there is an $M$-forwarding path $Q$ from $u_1$ to $w$. 
Trace $Q$ from $u_1$, and let $x$ be the first vertex in $(X\cup \{v_1\})\setminus \{u_1\}$ that we encounter; 
such a vertex  exists because $V(G_2)\subseteq X$ holds. 
\begin{cclaim}
Without loss of generality, we can assume that  $x\neq v_1$, that is, $x\in X$ holds 
and $u_1Qx$ is a path with $v_1\not\in V(u_1Qx)\setminus\{u_1\}$, 
 which is $M$-forwarding from $u_1$ to $x$. 
\end{cclaim}
\begin{proof}
Suppose the claim fails, that is, $x = v_1$ holds.  
Then, $u_1\neq v_1$ holds by the definition of $x$.  
If $u_1Qv_1$ is an $M$-saturated path, 
then $P_1 + u_1Qv_1$ forms an $M$-alternating circuit that contains the non-allowed edges 
in $E(P_1) \cap \delta(G_1)$,   
which  contradicts  Lemma~\ref{lem:allowed}.  
Otherwise, that is, if $u_1Qv_1$ is an $M$-forwarding path from $u_1$ to $v_1$,  
then $v_1Qw$ is an $M$-forwarding path from $v_1$ to $w$ that is   disjoint from $u_1$.  
Redefine $x$  as the first vertex in $X$ that we encounter if we trace $v_1Qw$ from $v_1$. 
Then, $v_1Qx$ is a path that is disjoint from $u_1$ and is $M$-forwarding from $v_1$ to $x$,  
according to Lemma~\ref{lem:cut2forwarding}.  
Therefore, by swapping the roles of $u_1$ and $v_1$, 
without loss of generality, 
we obtain this claim. 
\qed
\end{proof} 

Therefore, hereafter let $x \in  X$ and let  
$u_1Qx$ be an $M$-forwarding path from $u_1$ to $x$ with $v_1\not\in V(u_1Qx)\setminus\{u_1\}$.

As $x\in X$ holds, $\parcondxp{X}{P}{G_1}{M}{G}$ implies that 
there is an $M$-forwarding path $R$ from $x$ to an internal vertex of $P_1$, say, $y$, 
such that $V(R)\subseteq X$ and $V(R)\cap V(P_1) = \{y\}$.

If $u_1P_1y$ has an even number of edges, 
then $u_1Qx + xRy + yP_1u_1$ is an $M$-alternating circuit that contains some non-allowed edges, 
say, the edges in $E(P_1)\cap \delta(u_1)$, 
which contradicts  Lemma~\ref{lem:allowed}. 

Hence hereafter we assume that  $u_1P_1y$ has an odd number of edges.
From Proposition~\ref{prop:nonpositive}, 
there is a path $L$ of $G_1$ that is $M$-saturated between $v_1$ and $u_1$ 
or $M$-forwarding from $v_1$ to $u_1$.  
Trace $L$ from $v_1$,  and let $z$ be the first vertex on $u_1Qx$; 
note that Lemma~\ref{lem:cut2forwarding} implies that $v_1Lz$ is an $M$-forwarding path from $v_1$ to $z$. 
Additionally, note that $L$ is disjoint from $X$,  
because $V(L)\subseteq V(G_1)$ holds and $X$ is disjoint from $V(G_1)$.  
If $u_1Qz$ has an odd number of edges, then
$zQu_1 + P_1 + v_1Lz$ is an $M$-alternating circuit that contains  non-allowed edges, 
say, the edges in $E(P_1)\cap \delta(G_1)$, which  contradicts  Lemma~\ref{lem:allowed}. 
If  $u_1Qz$ has an even  number of edges,
then $v_1Lz+ zQx + xRy + yP_1u_1$ is an $M$-alternating circuit,
which is again a contradiction.
Thus we obtain $G_1 = G_2$, and the theorem is proved.
\qed
\end{proof}

\begin{remark}
The partially ordered structure determined by $\yield$ is not a generalization 
of the Dulmage-Mendelsohn decomposition. 
We can confirm that $\yield$ is  determined totally unique for a  graph, 
whereas the partial order for the Dulmage-Mendelsohn decomposition 
depends on the choice of color classes. 
If a graph $G$ is bipartite, 
then the poset $(\mathcal{G}(G), \yield)$ 
has a trivial structure.   
In our next work~\cite{kita2012c}, 
we give a generalization of the Dulmage-Mendelsohn decomposition 
using the results in this paper. 
\end{remark}

\begin{remark} 
From Lemma~\ref{lem:order2const}, 
any inconsistent factor-component is minimal with respect to $\yield$.  
\end{remark} 

We call the partial order $\yield$  the {\em basilica order}. 


\section{Generalization of the Kotzig-Lov\'asz decomposition} \label{sec:part}

In this section, we give a generalization of the Kotzig-Lov\'asz decomposition 
for general  graphs. 
Given a graph $G$, 
the {\em deficiency} of $G$ is the number $|V(G)|-2\nu(G)$ 
and is denoted by $\deficiency{G}$, 
where $\nu(G)$ is the size of a maximum matching. 
That is, $\deficiency{G}$ is the number of vertices exposed by a maximum matching.   
Note that $\deficiency{G} = 0$ if and only if $G$ is factorizable. 
\begin{definition} 
Let $G$ be a graph. 
For $u,v\in V(G)$,  
we say $u\gsim{G} v$ 
if $u = v$ holds or 
if $u$ and $v$ are contained in the same factor-component 
and $\deficiency{G-u-v} > \deficiency{G}$ holds. 
\end{definition}

By  definition, if a  graph $G$ is consistently factor-connected then 
the binary relation $\gsim{G}$ coincides with $\sim$ given by Kotzig~\cite{kotzig1959a, kotzig1959b, kotzig1960}. 
We prove in Theorem~\ref{thm:generalizedcanonicalpartition} that $\gsim{G}$ is an equivalence relation. 
Lemmas~\ref{lem:d2single} to \ref{lem:a2sim} in the following are used to prove Theorem~\ref{thm:generalizedcanonicalpartition}. These lemmas relate the deficiency and the Gallai-Edmonds family, 
and can be observed more easily from the Gallai-Edmonds structure theorem. 
However, we prove them without using the Gallai-Edmonds structure theorem 
 to keep our results  self-contained.  

The next lemma  implies that each vertex in $D(G)$ forms an equivalence class that is a singleton. 
\begin{lemma} \label{lem:d2single}
Let $G$ be a graph, and let $u\in D(G)$. 
Then, for any $v\in V(G)\setminus \{u\}$, 
$\deficiency{G-u-v} \le \deficiency{G}$ holds. 
\end{lemma} 

\begin{proof} 
As $u\in D(G)$ holds,   $\deficiency{G-u} = \deficiency{G}-1$. 
Obviously, $|\deficiency{G-u-v} - \deficiency{G-u}| = 1$. 
Hence, $\deficiency{G-u-v} \le \deficiency{G-u} + 1 = \deficiency{G}$. 
\qed
\end{proof}

The next lemma will be used in both Lemma~\ref{lem:a2sim}   and Theorem~\ref{thm:generalizedcanonicalpartition}. 
\begin{lemma}\label{lem:def2saturated} 
Let $G$ be a graph and $M$ be a maximum matching of $G$, 
and let $u, v\in V(G)\setminus D(G)$ be two distinct vertices. 
Then  $\deficiency{G-u-v} \le \deficiency{G}$ holds if and only if 
there exists an $M$-saturated path between $u$ and $v$. 
\end{lemma} 

\begin{proof} 
We first prove the necessity. 
Let $P$ be an $M$-saturated path between $u$ and $v$. 
Then, $M\triangle E(P)$ is a matching of $G-u-v$ that covers any vertex that $M$ covers 
other than $u$ and $v$. 
Hence, $\deficiency{G-u-v} \le \deficiency{G}$ holds. 

Next, we prove the sufficiency. 
If $uv$ is an edge in $M$, then the claim obviously holds. 
Hence, in the following, we assume $uu', vv'\in M$ for some $u', v' \in V(G)\setminus \{u,v\}$. 
Then, $M\setminus \{uu', vv'\}$ is a matching of $G-u-v$ 
but is not maximum, because it exposes the vertices in  $S \cup \{u', v'\}$, 
where $S$ is the set of vertices that $M$ exposes. 
Hence, $G-u-v$ has an $M$-exposed path $P$ whose ends are in $S \cup \{u', v'\}$. 
 If both ends are in $S$, 
 then $M\triangle E(P)$ is a bigger matching of $G$ than $M$, which is a contradiction. 
 If one end $x$ is in $S$ and the other is equal to either $u'$ or $v'$, say, $u'$, 
 then, in $G$, $P + uu'$ is an $M$-forwarding path from $u$ to $x$. 
 This implies  $u\in D(G)$ from Lemma~\ref{lem:da2path} \ref{item:da2path:d}, which is a contradiction. 
 Therefore, the ends of $P$ are $u'$ and $v'$, 
 and $P+ uu' + vv'$ is an $M$-saturated path between $u$ and $v$. 
\qed
\end{proof}

The next lemma  implies that $A(G)\cap V(H)$ forms an equivalence class 
for each $H\in\inconst{G}$. 
\begin{lemma} \label{lem:a2sim}
Let $G$ be a graph, and let $M$ be a maximum matching of $G$.  
For any $u,v\in A(G)$, 
$\deficiency{G-u-v} > \deficiency{G}$ holds. 
\end{lemma} 

\begin{proof} 
If $u=v$, then the claim obviously holds. 
Hence, assume $u\neq v$ and  
suppose the claim fails, that is, suppose $\deficiency{G-u-v} \le \deficiency{G}$. 
Then, by Lemma~\ref{lem:def2saturated}, 
there exists an $M$-saturated path $P$ between $u$ and $v$. 
By Lemma~\ref{lem:da2path} \ref{item:da2path:a}, 
there is an $M$-exposed path $Q$ between $u$ and a vertex $x$ exposed by $M$. 
Trace $Q$ from $x$, and let $y$ be the first encountered vertex in $V(P)$. 
Then, $xQy$ is an $M$-forwarding path from $x$ to $y$, 
and, for a vertex $w$ that is equal to either $u$ or $v$, 
the path $xQy + yPw$ is an $M$-forwarding path from $w$ to $x$. 
This implies $w\in D(G)$ according to Lemma~\ref{lem:da2path} \ref{item:da2path:d}, which is a contradiction. 
Hence, we obtain $\deficiency{G-u-v} > \deficiency{G}$. 
\qed
\end{proof}

The next theorem presents our generalization of the Kotzig-Lov\'asz decomposition. 

\begin{theorem} \label{thm:generalizedcanonicalpartition} 
For any  graph $G$, 
the binary relation $\gsim{G}$ is an equivalence relation on $V(G)$. 
\end{theorem} 
\begin{proof}
Reflexivity and symmetry obviously hold by definition. 
We  prove  transitivity in the following. 
Let $u, v, w\in V(H)$ be such that $u\gsim{G} v$ and  $v\gsim{G} w$. 
If any two among $u, v, w$ are identical, clearly the claim follows.
Therefore, it suffices to consider the case that they are mutually distinct. 
If $H$ is inconsistent, then, from Lemma~\ref{lem:d2single}, 
$u, v, w \in A(G)$ follows. 
Thus, from Lemma~\ref{lem:a2sim}, $u\gsim{G} w$ is obtained. 
Therefore, in the remainder of this proof, we assume that  $H$ is consistent. 
Suppose that the claim fails, that is, $u\not\gsim{G} w$.
From Lemma~\ref{lem:def2saturated},  there is an $M$-saturated path $P$
between $u$ and $w$.
By Proposition~\ref{prop:nonpositive}, there is an $M$-\zero path $Q$ from $v$ to $u$.
Trace $Q$ from $v$ and let $x$ be the first vertex we encounter  in $V(Q)\cap V(P)$.
If $uPx$ has an odd number of edges, then 
$vQx + xPu$ is an $M$-saturated path between $u$ and $v$, which is a contradiction.
If $uPx$ has an even  number of  edges,
then $xPw$ has an odd number of edges, and by the same argument we have a contradiction.
\qed
\end{proof}

If a graph $G$ is consistently factor-connected, then 
the family of equivalence classes under $\gsim{G}$, that is, $V(G)/\gsim{G}$,   coincides with 
the original Kotzig-Lov\'asz decomposition~\cite{kotzig1959a, kotzig1959b, kotzig1960}.  
Therefore, 
for a general graph $G$, 
we denote $V(G)/\gsim{G}$ by $\gpart{G}$, 
and call it 
the {\em generalized Kotzig-Lov\'asz decomposition} or simply 
the {\em Kotzig-Lov\'asz decomposition}. 
By the definition of $\gsim{G}$, 
 each equivalence class is contained 
in some factor-component. 
Therefore, for each $H\in\mathcal{G}(G)$, 
the family $\{ S\in \gpart{G} : S\subseteq V(H)\}$ is a partition of $V(H)$; 
we denote this partition by $\pargpart{G}{H}$. 

The next statement shows that our generalization of the Kotzig-Lov\'asz decomposition 
provides  information that the classical Kotzig-Lov\'asz decomposition 
does not. 

\begin{observation} \label{prop:refinement} 
For a factorizable graph $G$ and a factor-component $H\in \mathcal{G}(G)$, 
the partition $\pargpart{G}{H}$ is a refinement of $\gpart{H}$; 
that is, 
if two vertices $u, v\in V(H)$ satisfy $u \gsim{G} v$ in $G$, 
then $u\sim v$ holds in $H$. 
\end{observation}

In general, $\pargpart{G}{H}$ is a proper refinement of $\gpart{H}$. 
Therefore, our generalization of the Kotzig-Lov\'asz decomposition 
is not  trivial; that is, 
$\gpart{G}$ is not merely a  disjoint union of 
the Kotzig-Lov\'asz decomposition of each factor-component.  

Our proof of Theorem~\ref{thm:generalizedcanonicalpartition}  provides 
a shortened and self-contained proof of the classical Kotzig-Lov\'asz decomposition. 
Kotzig's proof consists of three papers, so 
 proving that  $\sim$ is an equivalence relation from first principles  
has been considered challenging~\cite{lp1986}.
Lov\'asz's proof uses the Gallai-Edmonds structure theorem,  
and, accordingly, is not self-contained. 
However, in fact, it can be proved in a  simple  way
even without the premise of the Gallai-Edmonds structure theorem  or the notion of barriers. 
All the results used to obtain Theorem~\ref{thm:generalizedcanonicalpartition} 
are self-contained in this paper.

Lov\'asz~\cite{lovasz1972b} reformulated the classical Kotzig-Lov\'asz decomposition using the notion of {\rm barriers}~\cite{lp1986}. 
In our next paper~\cite{DBLP:conf/cocoa/Kita13, kita2012canonical}, 
we discuss the relationship between barriers and our generalized Kotzig-Lov\'asz decomposition, 
and show that our decomposition also provides a generalization of Lov\'asz's formulation. 

The next observation follows from Lemmas~\ref{lem:d2single} and \ref{lem:a2sim}. 
\begin{observation} \label{note:inconstpart}
Let $G$ be a graph, and let $H\in\inconst{G}$. 
Then, $\pargpart{G}{H} = \{A(G)\cap V(H)\} \cup \bigcup\{ \{x\} : x\in D(G)\cap V(H)\}$. 
\end{observation} 
\begin{observation} 
Let $G$ be a graph. 
For a factor-component $H\in\comp{G}$, 
$\pargpart{G}{H}$ consists of only a single member
if and only if $| V(H) | = 1$, which implies that its only vertex is in $D(G)$.  
\end{observation} 


\begin{remark} 
An alternative way to define $\gsim{G}$ is the following. 
Given a graph $G$, for  any $u,v\in V(G)\setminus D(G)$, 
we say $u\gsim{G} v$ if $u$ and $v$ are contained in the same factor-component 
and $\deficiency{G-u-v} > \deficiency{G}$ holds.  
Obviously, $\gsim{G}$ is also an equivalence relation over $V(G)\setminus D(G)$, 
and its equivalence classes coincide with those given in this section, 
except for the trivial classes over $D(G)$.   
We prefer  this formulation for the generalized Kotzig-Lov\'asz decomposition given 
the nature of matchings shown in our next paper~\cite{DBLP:conf/cocoa/Kita13, kita2012canonical}. 
However, in this paper, we employ the other formulation. 
\end{remark}

\section{Basilica Type Relationship 
and  Definition of New Canonical Decomposition} \label{sec:cor} 
\subsection{Relationship between $\yield$ and $\gsim{G}$} \label{sec:cor:cor} 
There is a relationship between 
 the partial order $\yield$ and the generalized Kotzig-Lov\'asz decomposition,  
 even though they are given independently.
We state this relationship in Theorem~\ref{thm:base}, 
using the following definitions and lemmas. 
\begin{definition} 
Let $G$ be a  graph, and let $H\in\mathcal{G}(G)$. 
We denote by $\parupstar{G}{H}$ the set of upper bounds of $H$ in the poset $(\mathcal{G}(G), \yield)$; 
that is,  
$\parupstar{G}{H} := \{ H'\in\mathcal{G}(G): H\yield H'\}$. 
We define 
$\parup{G}{H} := \parupstar{G}{H}\setminus \{H\}$  
and denote by $\vparupstar{G}{H}$ and  $\vparup{G}{H}$  the sets of vertices that are contained in 
the factor-components in $\parupstar{G}{H}$  and in $\parup{G}{H}$, respectively; 
that is, 
$\vparupstar{G}{H} := \bigcup_{H'\in\parupstar{G}{H}} V(H')$ and 
$\vparup{G}{H} := \bigcup_{H'\in\parup{G}{H}} V(H')$.  
We often omit the subscripts ``$G$'' 
if they are apparent from the context. 
\end{definition} 

\begin{lemma}\label{lem:nonrefinable} 
Let $G$ be a  graph and $M$ be a maximum matching of $G$, 
and let $G_1, G_2 \in \mathcal{G}(G)$ be distinct factor-components.  
If there exists an $M$-ear $P$ with $\earint{P} \subseteq C(G)$ that is relative to $G_1$ and traverses $G_2$, 
then $G_1\yield G_2$ holds. 
Accordingly, any factor-component traversed by $P$ is an upper bound of $G_1$. 
\end{lemma} 
\begin{proof} 
As stated in Lemma~\ref{lem:inductive-ear}, $\parcondxp{\earint{P}}{P}{G_1}{M}{G}$ holds. 
Thus, from Lemma~\ref{lem:compclosure}, using $\earint{P}$, 
we can construct a critical-inducing set for $G_1$ to $G_2$. 
Thus, $G_1\yield G_2$ holds, and accordingly the remaining statement is also obtained.  
\qed
\end{proof}

\begin{lemma}\label{lem:ear-base}
Let $G$ be a  graph and $M$ be a  matching of $G$,
and let $H\in \mathcal{G}(G)$.
Let $P$ be an $M$-ear relative to $H$ with end vertices 
$u, v \in V(H)$ and with $\earint{P}\subseteq C(G)$.
Then $u\gsim{G} v$ holds.
\end{lemma}
\begin{proof}
First, note $\earint{P}\subseteq \vup{H}$ according to Lemma~\ref{lem:nonrefinable}. 
Hence, if $H$ is an inconsistent factor-component, then $u, v\in A(G)\cap V(H)$ holds. 
Therefore, we obtain $u\gsim{G} v$ from Lemma~\ref{lem:a2sim}. 

Hence, in the following, assume that $H$ is consistent. 
Suppose the claim fails, that is, $u \not\gsim{G} v$ holds. 
Then, from Lemma~\ref{lem:def2saturated}, there is an $M$-saturated path $Q$ between $u$ and $v$. 
Trace $Q$ from $u$, 
 and let $x$ be the first vertex we encounter that is in $V(P)\setminus \{u\}$. 
 If $x = v$, then $Q + P$ is an $M$-alternating circuit 
 that contains some non-allowed edges of $\parcut{G}{H}$, 
 which contradicts Lemma~\ref{lem:allowed}.   
 Hence, we assume $x\in \earint{P}\setminus \{u\}$ in the following. 
If $uPx$ has an even number of  edges, 
then $uQx + xPu$ is an $M$-alternating circuit with some non-allowed edges of $\parcut{G}{H}$, which
is again a contradiction. 
Hence, we  assume  that $uPx$ has an odd number of  edges.
Let $I\in \mathcal{G}(G)$ be the factor-component that contains $x$.  
The  connected components of $uQx + xPu - E(I)$ 
are  $M$-ears relative to $I$, 
and one of them traverses $H$. 
This implies $I\yield H$ under Lemma~\ref{lem:nonrefinable}, 
which contradicts $H\yield I$ under Theorem~\ref{thm:order}. 
\qed
\end{proof}

\begin{lemma} \label{lem:base} 
Let $G$ be a  graph 
and $M$ be a maximum matching of $G$,  
and let $G_0\in\mathcal{G}(G)$.  
Let $X\subseteq C(G)$ be a set of vertices 
such that $\parcondxp{X}{P}{G_0}{M}{G}$ holds for some $M$-ear $P$ relative to $G_0$.  
Then, 
\begin{enumerate} 
\renewcommand{\labelenumi}{\theenumi}
\renewcommand{\labelenumi}{{\rm \theenumi}}
\renewcommand{\theenumi}{(\roman{enumi})}
\item \label{item:k} 
there exists a connected component $K$ of $G[\vup{G_0}]$ 
with $X \subseteq V(K)$; and, 
\item \label{item:n}
there exists $T\in\pargpart{G}{G_0}$ such that 
$\parNei{G}{X} \cap V(G_0) \subseteq T$ holds. 
\end{enumerate} 
\end{lemma} 
\begin{proof} 
As Lemma~\ref{lem:compclosure} states that $X$ is contained in a critical-inducing set for $G_0$, 
we have $X\subseteq \vup{G_0}$. 
Additionally, $\parcondxp{X}{P}{G_0}{M}{G}$ implies 
that $G[X]$ is connected. 
Therefore, \ref{item:k} follows. 

Let $u$ and $v$ be the ends of $P$. 
From Lemma~\ref{lem:ear-base}, 
there exists $T\in\pargpart{G}{G_0}$ with $\{u, v\} \subseteq T$. 
Let $w \in \parNei{G}{X}\cap V(G_0)$, and  let $z\in X$ 
be  a vertex  with $wz\in E(G)$. 
From Lemma~\ref{lem:int2root}, 
there exists an $M$-forwarding path $Q$ from $z$ to $r\in \{u, v\}$ 
with $V(Q)\setminus \{r\} \subseteq X$. 
Then, $wz + Q$ forms an $M$-ear relative to $G_0$ 
whose ends are $w$ and $r$. 
Therefore, from Lemma~\ref{lem:ear-base}, 
$w\in T$ follows, and we have \ref{item:n}. 
\qed
\end{proof}

The relationship between the basilica order and 
 generalized Kotzig-Lov\'asz decomposition is shown in the next theorem.   
\begin{theorem}\label{thm:base}
Let $G$ be a  graph, 
and let $G_0\in\mathcal{G}(G)$. 
For each connected component $K$ of $G[\vup{G_0}]$,  
there exists $T_K\in\pargpart{G}{G_0}$ such that 
$\parNei{G}{K}\cap V(G_0)\subseteq T_K$. 
\end{theorem}
\begin{proof}
Let $M$ be a  maximum matching of $G$. 

Define a family $\mathcal{X}\subseteq 2^{V(K)}$ as follows: 
$X\subseteq V(K)$ is a member of $\mathcal{X}$ if 
 $\parcondxp{X}{P}{G_0}{M}{G}$ holds
 for some $M$-ear $P$ relative to $G_0$. 
\begin{cclaim} 
It holds that $\bigcup_{X\in\mathcal{X}} X = V(K)$. 
\end{cclaim} 
\begin{proof} 
From the definition of $\mathcal{X}$, clearly $\bigcup_{X\in\mathcal{X}} X \subseteq V(K)$. 
In contrast,  $V(K)$ is obviously separating, and, 
from Lemma~\ref{lem:order},  
each factor-component that composes $V(K)$ 
is contained in a set of vertices $X$ with $\parcondxp{X}{P}{G_0}{M}{G}$ 
for some $M$-ear $P$ relative to $G_0$;  
from Lemma~\ref{lem:base}, this $X$ satisfies $X\subseteq V(K)$. 
Therefore, we have $\bigcup_{X\in\mathcal{X}} X \supseteq V(K)$. 
\qed
\end{proof}

For each $T\in\pargpart{G}{G_0}$, 
we define $\mathcal{X}_T \subseteq \mathcal{X}$ as follows: 
$X\in\mathcal{X}$ is a member of $\mathcal{X}_T$ if 
 $\parNei{G}{X}\cap V(G_0) \subseteq T$ holds. 
From Lemma~\ref{lem:base}, 
if $S\neq T$, then $\mathcal{X}_S \cap \mathcal{X}_T = \emptyset$; additionally,  
$\bigcup_{T\in\pargpart{G}{G_0}} \mathcal{X}_T = \mathcal{X}$ holds. 

\begin{cclaim} \label{claim:disjoint} 
Let $S, T\in \pargpart{G}{G_0}$. 
Let $X \in \mathcal{X}_S$ and $Y \in \mathcal{X}_T$. 
If $X \cap Y \neq \emptyset$, then $S = T$. 
If $X \cap Y = \emptyset$ and $E[X, Y] \neq \emptyset$, 
then $S = T$.  
\end{cclaim} 
\begin{proof} 
First assume $X \cap Y \neq \emptyset$. 
As both $X$ and $Y$ are closed with respect to $M$, 
so is $X\cap Y$. 
Take $x\in X\cap Y$ arbitrarily; 
from Lemma~\ref{lem:int2root}, 
we have an $M$-forwarding path $Q$ from $x$ to a vertex $r\in V(G_0)$ 
with $V(Q)\setminus \{r\} \subseteq X$; 
from Lemma~\ref{lem:base}, we have $r\in S$.  
Trace $Q$ from $r$, and let $y$ be the first vertex we encounter 
that is in $\parNei{G}{X\cap Y}$;   
let $z\in X\cap Y$ be such that $yz\in E[X\setminus Y, X\cap Y]$. 
Here, $rQy$ is an $M$-forwarding path with 
$V(rQy)\setminus \{r\} \subseteq X\setminus Y$. 
By contrast, 
we also have an $M$-forwarding path $R$ from $z$ to a vertex $s\in T$ 
with $V(R)\setminus \{s\} \subseteq Y$. 
Here, $rQy + yz + R$ is an $M$-ear relative to $G_0$ with ends $r$ and $s$. 
From Lemma~\ref{lem:ear-base}, $S = T$ follows. 

Next, assume $X \cap Y = \emptyset$ and $E[X, Y] \neq \emptyset$. 
Let $t_1\in X$ and $t_2 \in Y$ be vertices with $t_1t_2\in E[X, Y]$. 
From Lemma~\ref{lem:int2root}, 
for each $i\in\{1,2\}$, 
we have an $M$-forwarding path $L_i$ from $t_i$ 
to a vertex $r_i\in V(G_0)$ with $V(L_1)\setminus\{r_1\}\subseteq X$ 
and $V(L_2)\setminus\{r_2\}\subseteq Y$; 
from Lemma~\ref{lem:base}, 
we have $r_1\in S$ and $r_2\in T$.  
Therefore, 
$L_1 + t_1t_2 + L_2$ forms an $M$-ear relative to $G_0$ with ends  $r_1$ and $r_2$. 
From Lemma~\ref{lem:ear-base}, again $S = T$ follows. 
\qed
\end{proof}

As $K$ is connected, 
Claim~\ref{claim:disjoint} implies $|\{ T\in\pargpart{G}{G_0} : \mathcal{X}_T \neq \emptyset\}| = 1$.  
This completes the proof.

\qed
\end{proof}


\subsection{Declaration of New Canonical Decomposition} \label{sec:cor:dec} 
We can now declare a new canonical decomposition  
in which the basilica order and the generalized Kotzig-Lov\'asz decomposition are unified 
through Theorem~\ref{thm:base}.  
According to  Theorem~\ref{thm:base},
the strict upper bounds on a factor-component 
are each ``attached'' or ``assigned'' to an equivalence class of the generalized Kotzig-Lov\'asz decomposition. 
That is, 
let $H$ be a factor-component of a graph $G$,  
let $I\in\comp{G}\setminus \{H\}$ be such that $H\yield I$, 
and let $K$ be the connected component of $G[\vup{H}]$ with $V(I)\subseteq V(K)$. 
If $S\in\pargpart{G}{H}$ is such that $\parNei{G}{K}\cap V(H)\subseteq S$ 
 as in Theorem~\ref{thm:base}, 
then we can view  $I$ as being ``attached'' or ``assigned'' to $S$ as an upper bound on $H$. 
Hence, 
a graph can be regarded as being constructed 
by repeatedly assigning and attaching 
each factor-component to  an equivalence class possessed by  a lower bound.

Although the basilica order structure 
and the generalized Kotzig-Lov\'asz decomposition themselves 
can  be  considered individually as  canonical decompositions, 
they are integrated into a single theory of a canonical decomposition 
 through the relationship given by Theorem~\ref{thm:base}. 
We call this integrated concept 
the {\em basilica decomposition}, 
because this evokes the idea of a graph being structured like an architectural building. 
The term ``basilica'' comes from 
the {\em cathedral theorem}  by Lov\'asz~\cite{lovasz1972b, lp1986}, 
which is an inductive characterization of {\em saturated graphs}. 
 In fact, the cathedral theorem can be derived from our new canonical decomposition~\cite{kita2014alternative}.

\section{Inconsistent Factor-components Via Gallai-Edmonds Structure Theorem} 
In this section, we use the Gallai-Edmonds structure theorem to obtain futher information about the inner structure of inconsistent factor-components. 
\begin{lemma}\label{lem:forwarding2allowed} 
Let $G$ be a graph, $M$ be a maximum matching of $G$, 
and $r$ be a vertex exposed by $M$. 
If $P$ is an $M$-forwarding path from some vertex to $r$, 
then all edges of $P$ are allowed 
and therefore $P$ is contained in a  factor-component. 
\end{lemma} 

\begin{proof} 
If $P$ is such a path, then $M\triangle E(P)$ is also a maximum matching of $G$. 
Hence, the claim follows. 
\qed
\end{proof} 

\begin{lemma}\label{lem:d2comp}  
Let $G$ be a graph. 
If $K$ is a connected component of $G[D(G)]$,  then
the vertices in $V(K)\cup \parNei{G}{K}$ are contained in the same factor-component. 
\end{lemma} 

\begin{proof}
According to Theorem~\ref{thm:gallaiedmonds}, $K$ is factor-critical. 
Let $r\in V(K)$, 
and let $M$ be a maximum matching of $G$ exposing $r$. 
Arbitrarily choose $x\in V(K)$. 
From Lemma~\ref{lem:path2root},  
there is an $M$-forwarding path $P$ from $x$ to $r$. 
From Lemma~\ref{lem:forwarding2allowed}, 
$x$ and $r$ are contained in the same factor-component. 
Thus, all vertices of $K$ are contained in the same factor-component. 
From Theorem~\ref{thm:gallaiedmonds}, 
any edge in $\parcut{G}{K}$ is allowed. 
Therefore, the vertices in $\parNei{G}{K}$ are also contained in the same factor-component 
as the vertices of $K$.    
\qed
\end{proof}

The next theorem follows from Lemma~\ref{lem:d2comp} and Theorem~\ref{thm:gallaiedmonds}. 
\begin{theorem}\label{thm:inconst2connected}  
Let $G$ be a graph. 
Any subgraph $H$ is an inconsistent factor-component of $G$ 
if and only if it is a connected component of $G[D(G)\cup A(G)]\setminus E(G[A(G)])$. 
\end{theorem}

\section{Pertinent Properties}\label{sec:pertinentprops}
\subsection{Non-triviality of $\yield$} \label{sec:add}
The following theorem shows that 
most factorizable graphs with more than one factor-components 
have non-trivial structures as posets. 
\begin{theorem}\label{thm:add}
Let $G$ be a factorizable graph,
$G_1, G_2 \in \mathcal{G}(G)$ be
 factor-components for which  $G_1\yield G_2$ does not hold,
  and 
let $G_1$ be minimal in the poset $(\mathcal{G}(G), \yield)$.
Then there are possibly identical complement edges $e$ and  $f$ of $G$  between
$G_1$ and $G_2$ with 
$\mathcal{G}(G + e + f)  = \mathcal{G}(G)$ and 
$G_1\yield G_2$ in $(\mathcal{G}(G+e+f), \yield)$.
\end{theorem}

\begin{proof}
First, we prove  the case where there is an edge $xy$  with $x\in V(G_1)$ and $y\in V(G_2)$.
Let $M$ be a perfect matching of $G$.
Choose a vertex $w\in V(G_2)$  with  $w\not\sim y$ in $G_2$,
and let $P$ be an $M$-saturated path of $G_2$ between $w$ and $y$.
If $xw\in E(G)$ holds, then 
$xy + P + wx$ is an $M$-ear that is relative to $G_1$ and traverses $G_2$. 
This implies $G_1\yield G_2$ under Lemma~\ref{lem:nonrefinable}, 
which is a contradiction.
Thus, $xw\not\in E(G)$ holds.

Suppose $\mathcal{G}(G+xw) \neq \mathcal{G}(G)$.
Then, Lemma~\ref{lem:allowed} implies that $G + xw$ has an $M$-alternating circuit that contains $xw$, hence $G$ has an $M$-saturated path $C$ between $x$ and $w$.  
Trace $C$ from $x$,   
and let $z$ be the first vertex   in $V(G_2)$ that we encounter. 
Then, $xy + xCz$ is an $M$-ear of $G$ that is relative to $G_2$
and traverses $G_1$, which implies $G_2\yield G_1$ under Lemma~\ref{lem:nonrefinable};  
this contradicts the minimality of $G_1$. 
Thus, $\mathcal{G}(G+xw) = \mathcal{G}(G)$,  and  we   have proved  this case.

We now consider the other case,
where no edge of $G$ connects $G_1$ and $G_2$.
Choose $x\in V(G_1)$ and $y\in V(G_2)$ arbitrarily.
If $\mathcal{G}( G + xy ) = \mathcal{G}(G)$ holds, then  
we can reduce it to the first case and the claim follows.

Therefore, it suffices to consider the case with $\mathcal{G}(G + xy ) \neq \mathcal{G}(G)$. 
Then, from Lemma~\ref{lem:allowed}, for any perfect matching $M$ of $G$,
 $G+xy$  has an $M$-alternating circuit that  contains $xy$. 
Thus, we have an $M$-saturated path $C$ between $x$ and $y$ in $G$. 
Trace $C$ from $y$,  
and let $u$ be the first vertex in  $G_1$ that we encounter.  
Furthermore, trace $uCy$ from $u$, and let $v$ be the first vertex we encounter that is in $G_2$.

If $\mathcal{G}(G + uv) = \mathcal{G}(G)$, then the claim follows by the same argument.

Otherwise, that is, if  $\mathcal{G}(G + uv) \neq \mathcal{G}(G)$, then 
 Lemma~\ref{lem:allowed} implies that  
$G$ has an $M$-alternating circuit that contains $uv$. 
Thus, we have an $M$-saturated path $D$ between $u$ and $v$ in $G$. 

Trace $D$ from $u$,   
and  let $w$ be the first vertex of $vCu - u$ that we encounter. 

If $wCu$ has an even number of edges, 
then $wCu + uDw$ is an $M$-alternating circuit of $G$ that contains non-allowed edges, 
which is a contradiction according to Lemma~\ref{lem:allowed}. 
Therefore, we assume that $wCu$ has an odd number of edges.  
Let  $H\in\mathcal{G}(G)$ be the factor-component with $w\in V(H)$.

Then, $wCu + uDw - E(H)$ is  an $M$-ear  that  is relative to $H$ and traverses $G_1$; 
this implies $G_1\yield H$ from Lemma~\ref{lem:nonrefinable}, 
which contradicts the minimality of $G_1$. 
Thus, this completes the proof. 
\qed
\end{proof}

\subsection{Vertices in Upper Bounds} 
From Theorem~\ref{thm:base}, 
the following is derived rather easily. 

\begin{theorem}\label{thm:maximumup} 
Let $G$ be a  graph, and let $H\in\mathcal{G}(G)$. 
Then, $\vupstar{H}$ is the maximum critical-inducing set for $H$; 
that is, the union of all the critical-inducing sets for $H$ 
is also a critical-inducing set for $H$ 
and equals $\vupstar{H}$. 
\end{theorem}  
\begin{proof} 
By Theorem~\ref{thm:order}, 
any critical-inducing set for $H$ is contained in $\vupstar{H}$. 
Therefore, by Lemma~\ref{lem:union}, 
the union of all the critical-inducing set for $H$ is 
also a critical-inducing set for $H$, contained in $\vupstar{H}$. 

Conversely, 
by the definition of $\upstar{H}$,  
for each $I\in\upstar{H}$, 
there is a critical-inducing set for $H$ to $I$. 
Therefore, $\vupstar{H}$ is contained in, 
and accordingly coincides with 
the union of all the critical-inducing sets. 
\qed
\end{proof}

Thus, we have the following as a corollary of Theorem~\ref{thm:maximumup}. 

\begin{corollary}\label{cor:2fc}
Let $G$ be a  graph, and let $H\in\mathcal{G}(G)$ 
and $S\subseteq \pargpart{G}{H}$. 
Let $K_1,\ldots, K_l$, where $l \ge 1$,  be the connected components of $G[\vup{H}]$ 
such that $\Gamma(K_i)\cap V(H)\subseteq S$ for each $i \in \{ 1,\ldots, l\}$. 
Then, $G[ V(K_1)\cup\cdots\cup V(K_l) \cup S]/S$ is factor-critical. 
\end{corollary}

\subsection{Immediate Compatible Pair of Factor-Components} 

\begin{lemma} \label{lem:increment}
Let $G$ be a graph and $M$ be a maximum matching of $G$. 
Let $X\subseteq V(G)$ be an critical-inducing set for $G_1\in \comp{G}$ 
and let $P$ be an $M$-ear relative to $X$ with 
$\earint{P}\neq\emptyset$ and $\earint{P}\subseteq C(G)$. 
Let $Y := X \cup V(H_1)\cup \cdots \cup V(H_k)$, 
where $H_1,\ldots, H_k$ are the factor-components that $P$ traverses. 
Then, $Y$ is a critical-inducing set for $G_1$. 
\end{lemma} 
\begin{proof} 
According to Lemma~\ref{lem:path2base}, for each $v\in X$, there is an $M$-forwarding path $R_v$ from $v$ to a vertex in $G_1$. 
From Lemma~\ref{lem:extension}, 
$\parcondxp{\earint{P}}{P}{X}{M}{G}$ holds. 
Furthermore, from Lemma~\ref{lem:compclosure}, 
$\parcondxp{Y\setminus X}{P}{X}{M}{G}$ holds.  
Hence, 
for each $x\in Y\setminus X$, there is an $M$-forwarding path $L_x$ from $x$ to 
a vertex $w$ that is equal to one of the ends of $P$, 
according to Lemma~\ref{lem:int2root}. 
Therefore, 
$L_{x} +  R_{w}$ is an $M$-forwarding path from $x$ to a vertex in $G_1$. 
Thus, the statement is proved by Lemma~\ref{lem:path2base}. 
\qed
\end{proof} 

\begin{lemma} \label{lem:intermediate}
Let $G$ be a graph and $M$ be a maximum matching of $G$. 
Let $X\subseteq V(G)$ be an critical-inducing set for $G_1\in \comp{G}$ 
and let $P$ be an $M$-ear relative to $X$ with ends $u_1$ and $u_2$. 
Then, there are a factor-component $H$ with $V(H)\subseteq X$  
and an $M$-ear $Q$ relative  to $H$ such that $E(Q)\setminus E(G[X]) = E(P)$. 
\end{lemma} 
\begin{proof} 
Under Lemma~\ref{lem:path2base}, 
for each $i\in \{1,2\}$, 
there is an  $M$-forwarding path $Q_i$ from $u_i$ to some vertex in $V(G_1)$. 
Trace $Q_2$ from $u_2$, and let $x$ be the first encountered vertex in a factor-component $H$ that also has some vertices of $Q_1$; 
such $H$ certainly exists because $G_1$ has shares some verices with both $Q_1$ and $Q_2$. 
Furthermore, trace $Q_1$ from $u$, and let $z$ be the first vertex in $H$. 
Then, $H$ and $zQ_1u + P vQ_2y$ are desired factor-component and  $M$-ear. 
\qed
\end{proof} 

\begin{proposition} 
Let $G$ be a graph and $M$ be a maximum matching of $G$. 
Let $G_1$ and $G_2$ are distinct factor-components with $G_1\yield G_2$. 
If $G_1$ and $G_2$ are immediate, that is, for any $H\in \comp{G}$, $G_1\yield H \yield G_2$  implies $G_1 = H$ or $G_2 = H$, then there is an $M$-ear relative to $G_1$ that traverses $G_2$. 
\end{proposition} 
\begin{proof} 
Let $X$ be a  critical-inducing set $X$ for $G_1$ to $G_2$. 
Let $X'$ be a maximal (in fact, the maximum) subset of $X\setminus V(G_2)$ 
that is critical-inducing for $G_1$; 
such  $X'$ certainly exists because $V(G_1)\subseteq X\setminus V(G_2)$  is critical-inducing for $G_1$. 
 Under Lemma~\ref{lem:inductive-ear}, 
 there is an $M$-ear $P$ relative to $X'$ with $V(P)\subseteq X$ and $\earint{P} \neq \emptyset$. 
From Lemma~\ref{lem:increment}, 
 the minimum separating set that contains $X'\cup V(P)$ is a critical-inducing set for $G_1$. 
 Therefore, $P$ traverses $G_2$. 
Furthermore, Lemma~\ref{lem:intermediate} implies that 
we can use this $P$ to obtain an $M$-ear that is relative to a factor-component $H$ with $V(H)\subseteq X'$ 
and  traverses $G_2$. 
Therefore, from Lemma~\ref{lem:nonrefinable}, 
$G_1 \yield H$ and $H \yield G_2$ hold. 
As $H\neq G_2$ holds, we have $H = G_1$. 
This completes the proof of this proposition. 

\qed
\end{proof}

\section{Algorithmic Results} \label{sec:alg}
\subsection{Algorithmic Preliminaries} \label{sec:alg:pre} 
In the remainder of this paper, 
we present algorithms for computing the basilica decomposition. 
Section~\ref{sec:alg:pre} presents some preliminary facts that will be used in the remaining sections. 
We denote by $n$ and $m$ the numbers of vertices and edges of an input graph, respectively. 
Note that 
we can assume $m = \mathrm{\Omega}(n)$ and, accordingly, $O(n+m) = O(m)$ 
if an input graph is connected or factorizable. 
Section~\ref{sec:alg:comp} provides an algorithm for computing the factor-components, 
and then  
Sections~\ref{sec:alg:part} and \ref{sec:alg:order} present how to compute 
the generalized Kotzig-Lov\'asz decomposition  and the basilica order. 
Each costs  $O(nm)$ time, 
using Edmonds' maximum matching algorithm as a subroutine~\cite{edmonds1965}.   

\begin{theorem}[Micali and Vazirani~\cite{mv1980}, Vazirani~\cite{vazirani1994}]\label{thm:matchingalg}
Given a graph, one of its maximum matchings can be computed in $O(\sqrt{n}m)$ time. 
\end{theorem} 

The following two statements can be   found implicitly in Edmonds's algorithm~\cite{edmonds1965}. 
See also Lov\'asz and Plummer~\cite{lp1986}.

\begin{theorem}[implicitly stated in Edmonds~\cite{edmonds1965}]\label{dacalg} 
Given a graph $G$ and a maximum matching $M$, 
the set $D(G)$, $A(G)$, and $C(G)$ can be computed in $O(n+m)$ time. 
\end{theorem}

\begin{proposition}[implicitly stated in Edmonds~\cite{edmonds1965}]%
\label{prop:rootblossom} 
Let $G$ be a graph and $M$ be a maximum matching of $G$,  and let $r\in V(G)$ be a vertex exposed by $M$. 
Let $C$ be  the connected component of $G[D(G)]$ that contains $r$. 
\begin{rmenum}
\item Then, for any maximum matching $M'$ of $G$ that exposes $r$, 
$M'_X$ is a near-perfect matching of $C$. 
\item 
Define $\mathcal{X}\subseteq 2^{V(G)}$ as follows: 
$X\subseteq V(G)$ is a member of $\mathcal{X}$ if 
 $r\in X$ holds, 
$G[X]$ is factor-critical, and 
$M_X$ is a near-perfect matching of $G[X]$, exposing $r$. 
Then, the maximum member of $\mathcal{X}$ 
is equal to $V(C)$. 
\item 
Given $G$, $M$, and $r$, 
$C$ can be computed in $O(m)$ time. 
\end{rmenum} 
\end{proposition}

The next statement can be deduced from Edmonds' algorithm. 
See also Carvalho and Cheriyan~\cite{cc2005}. 

\begin{proposition}\label{prop:pathalg}
Let $G$ be a factorizable graph and $M$ be a perfect matching of $G$,
and let $u\in V(G)$.
\begin{rmenum}
\item  The set of vertices that can be reached from $u$ by 
an $M$-saturated path can be computed in $O(m)$ time.
\item All the allowed edges adjacent to $u$ can be computed in $O(m)$ time.
\item All the factor-components of $G$ can be computed in $O(nm)$ time.
\end{rmenum}
\end{proposition}

\subsection{Computing Factor-components} \label{sec:alg:comp} 
Propositions~\ref{prop:fcomp2dac} and \ref{prop:pathalg} show 
how to compute consistent factor-components, 
wheares Theorem~\ref{thm:inconst2connected} implies an algorithm for computing inconsistent factor-components. 
Hence, we now obtain the following: 
\begin{theorem} \label{thm:compalg}
Given a graph $G$, 
one of its perfect matchings $M$, 
and the sets $D(G)$, $A(G)$, and $C(G)$,   
the factor-components of $G$ are computed in $O(nm)$ time. 
\end{theorem} 

\begin{proof} 
Under Proposition~\ref{prop:fcomp2dac}, 
we can compute $\comp{G}$ 
by computing $\comp{G[D(G)\cup A(G)]}$ and $\comp{G[C(G)]}$ individually.  
From Theorem~\ref{thm:inconst2connected}, 
we can compute $\comp{G[D(G)\cup A(G)]}$ in $O(n + m)$ time. 
From Proposition~\ref{prop:pathalg}, 
we can compute $\comp{G[C(G)]}$ in $O(nm)$ time. 
Therefore, we can obtain $\comp{G}$ in $O(nm)$ time. 
\qed
\end{proof}

\subsection{Computing the Generalized Kotzig-Lov\'asz Decomposition}  \label{sec:alg:part}
From Observation~\ref{note:inconstpart}  
and Proposition~\ref{prop:pathalg}, 
we can    compute the generalized Kotzig-Lov\'asz decomposition.  
\begin{theorem} \label{thm:partalg}
Given a graph $G$, one of its maximum matchings $M$, the set of factor-components $\comp{G}$, 
and the sets $D(G)$, $A(G)$, and $C(G)$, 
the generalized Kotzig-Lov\'asz decomposition of $G$ can be computed in $O(nm)$ time. 
\end{theorem} 

\begin{proof}
We compute $\pargpart{G}{H}$ for each $H\in\comp{G}$. 
According to Observation~\ref{note:inconstpart}, 
if $H$ is inconsistent 
then $\pargpart{G}{H} = \{ V(H)\cap A(G) \} \cup \bigcup \{ \{x\} : x\in V(H)\setminus A(G)\}$. 
Therefore, $\pargpart{G}{H}$ for all $H\in\inconst{G}$ can be computed 
in $O(n)$ time in total. 

If $H$ is consistent,   
we can compute $\pargpart{G}{H}$ 
 in a similar way as  the Kotzig-Lov\'asz decomposition of 
a consistently factor-connected  graph~\cite{cc2005}.
That is, 
for each $v\in V(H)$,
compute the set of vertices $U$ that can be reached from $v$ by an 
$M$-saturated path,
and recognize $V(H)\setminus U$ as a member of $\pargpart{G}{H}$.
Each $U\in\pargpart{G}{H}$ can be computed in $O(m)$ time 
according to Proposition~\ref{prop:pathalg}. 
Therefore,  computing $\pargpart{G}{H}$ for all $H\in\const{G}$ 
costs $O(nm)$ time. 
Thus, the proof is completed. 
\qed
\end{proof}

\subsection{Computing the Basilica Order}  \label{sec:alg:order} 
In this section, 
we present an algorithm for computing the basilica order in $O(nm)$ time. 
We  determine the poset by computing $\vup{H}$ 
for each factor-component $H$. 
The following lemmas are provided 
to associate $\vup{H}$ with Proposition~\ref{prop:rootblossom}. 
 Lemmas~\ref{lem:nopath2comp} and \ref{lem:preserve} 
are used to prove Lemma~\ref{lem:vup2dcomp}. 
\begin{lemma} \label{lem:nopath2comp} 
Let $G$ be a graph and $M$ be a maximum matching of $G$. 
Let $H\in\comp{G}$. 
For no $x\in V(G)\setminus V(H)$ exposed by $M$ and for no  $y\in V(H)$, 
there exists an  $M$-exposed path from $x$ to $y$. 
\end{lemma} 

\begin{proof} 
Suppose this lemma fails, and let $P$ be an $M$-exposed path from $x\in V(G)\setminus V(H)$ to $y\in V(H)$. 
Then, $y$ is covered by $M$, 
because, otherwise $M\triangle E(P)$ would be a bigger matching of $G$ than $M$.  
Hence, there is a vertex $y'\in V(H)$ to which $y$ is matched by $M$. 
Then, $P+yy'$ is an $M$-forwarding path from $y'$ to $x$, 
and therefore, from Lemma~\ref{lem:forwarding2allowed}, 
$y$ and $x$ are contained in the same factor-component, which is a contradiction. 
\qed
\end{proof}

\begin{lemma} \label{lem:preserve} 
Let $G$ be a graph and $M$ be a maximum matching of $G$. 
Let $H\in\comp{G}$. 
Then, $M_{V(G)\setminus V(H)}$ is a maximum matching of $G/H$. 
\end{lemma} 

\begin{proof} 
Suppose  this lemma fails, that is,   $M_{V(G)\setminus V(H)}$ is not a maximum matching of $G/H$.
Then, $G/H$ has an $M$-exposed path $P$ in which  one end is the contracted vertex $h$ that corresponds to $H$
and the other end is a vertex exposed by $M$.   
In $G$, $P$ forms an $M$-exposed path between a vertex not in $V(H)$ to a vertex of $H$. 
This contradicts Lemma~\ref{lem:nopath2comp}. 
\qed
\end{proof}

The next lemma associates the set of strict upper bounds 
with the special subgraph $C$ depicted in Proposition~\ref{prop:rootblossom}. 

\begin{lemma}\label{lem:vup2dcomp}  
Let $G$ be a  graph and $M$ be a maximum matching of $G$, 
and let $G_0\in\mathcal{G}(G)$. 
Let $G' := G/G_0$, and let $g_0$ be the contracted vertex that corresponds to $G_0$. 
Then, there exists a connected component $C$ of $G'[D(G')]$ with $g_0\in V(C)$ 
such that $V(C)\setminus \{g_0\}$ is equal to $\vparup{G}{G_0}$. 
\end{lemma} 

\begin{proof} 
Let $M' := M \setminus E(G_0)$. 
Define $\mathcal{X}\subseteq 2^{V(G)}$ as follows:  
$X\subseteq V(G)$ is a member of $\mathcal{X}$ if $V(G_0)\subseteq X$ holds,  
$G[X]/G_0$ is factor-critical, and 
$M_X$ is a perfect matching of $G[X]$. 
Additionally, 
define $\mathcal{X}'\subseteq 2^{V(G')}$ as follows:  
$X'\subseteq V(G')$ is a member of $\mathcal{X}'$ if 
$g_0\in X'$ holds, $G'[X']$ is factor-critical, and $M'_{X'}$ is a near-perfect matching of $G'[X']$, 
exposing $g_0$. 
It is easy to see that 
 for $X\subseteq V(G)$ and $X'\subseteq V(G')$ with $X \setminus V(G_0) = X'\setminus \{g_0\}$, 
$X\in \mathcal{X}$ holds if and only if $X'\in \mathcal{X}'$ holds. 

According to Lemma~\ref{lem:preserve}, $M'$ is a maximum matching of $G'$, which exposes $g_0$. 
Hence, from Proposition~\ref{prop:rootblossom}, 
 there exists a connected component $C$ of $G'[D(G')]$ with $g_0\in V(C)$, 
 and $V(C)$ is equal to the maximum member of $\mathcal{X}'$. 
 Accordingly, 
$\mathcal{X}$ has the maximum member $X_0$, with  
 $X_0\setminus V(G_0) = V(C)\setminus \{g_0\}$.

In the following, we prove  $X_0  = \vparupstar{G}{G_0}$. 
From Proposition~\ref{prop:rootblossom}, with respect to any maximum matching of $G'$ that exposes $g_0$, 
$V(C)$  is closed. This implies that $X_0$ is a separating set of $G$. 
Accordingly, $X_0$ is a critical-inducing set for $G_0$, 
and therefore, from Theorem~\ref{thm:maximumup}, $X_0\subseteq \vparupstar{G}{G_0}$ holds. 
By contrast,  $X_0\supseteq \vparupstar{G}{G_0}$ holds, because 
$\vparupstar{G}{G_0} \in \mathcal{X}$ holds.  
Hence, we obtain $X_0  = \vparupstar{G}{G_0}$, 
and therefore, $\vparup{G}{G_0} = V(C)\setminus \{g_0\}$.  
\qed
\end{proof}

The next statement immediately follows from Proposition~\ref{prop:rootblossom} and Lemma~\ref{lem:vup2dcomp}. 
\begin{lemma}\label{lem:vupalg} 
Given a graph $G$, a maximum matching $M$  of $G$, 
and $H\in\comp{G}$, $\vup{H}$ can be computed in $O(m)$ time. 
\end{lemma} 

The next theorem shows how to compute the poset of the basilica order 
using Lemma~\ref{lem:vupalg}. 

\begin{theorem}\label{thm:orderalg}
Given a  graph $G$, one of its maximum matchings $M$, and $\mathcal{G}(G)$,
we can compute the poset $(\mathcal{G}(G), \yield)$ in $O(nm)$ time.
\end{theorem} 
\begin{proof} 
It is sufficient 
to list all the strict upper bounds for each factor-component of $G$
by the following procedure.  
\begin{algorithmic}[1]
\STATE Initialize $f: \mathcal{G}(G) \rightarrow 2^{\mathcal{G}(G)}$ 
by $f(H):= \emptyset$ for each $H\in\mathcal{G}(G)$; 
\FORALL{$H\in\mathcal{G}(G)$}
\STATE compute $\vup{H}$ according to Lemma~\ref{lem:vupalg}; 
\FORALL{$x\in \vup{H}$}
\STATE let $I\in\mathcal{G}(G)$ be such that $x\in V(I)$; 
\STATE $f(H) := f(H) \cup \{I\}$.  
\ENDFOR
\ENDFOR
\end{algorithmic}
The correctness of the algorithm is obvious.   
For each $H\in \mathcal{G}(G)$, the above procedure costs $O(m)$ time; 
therefore, the entire computation costs $O(nm)$ time. 
\qed
\end{proof}

\subsection{Concluding  Algorithms} \label{sec:alg:conclusion} 
From Theorems~\ref{thm:matchingalg}, \ref{dacalg},  and \ref{thm:compalg}, 
a maximum matching, the Gallai-Edmonds family, 
and the set of factor-components can be computed in $O(nm)$ time in total. 
Therefore, from Theorems~\ref{thm:partalg} and \ref{thm:orderalg}, 
we obtain an $O(nm)$ time algorithm for computing the basilica decomposition. 
\begin{theorem}
Given a graph $G$, 
the basilica order $\yield$ over $\comp{G}$ 
and the generalized Kotzig-Lov\'asz decomposition can be computed in 
$O(nm)$ time.
\end{theorem}

\section{Conclusion} \label{sec:conclusion} 
We have introduced a new canonical decomposition, the {\em basilica decomposition}. 
The central results that support our new theory 
are the {\em basilica order}, a canonical partial order over the set of factor-components 
(Theorem~\ref{thm:order}),  
the generalized Kotzig-Lov\'asz decomposition (Theorem~\ref{thm:generalizedcanonicalpartition}), 
and the  structure described by a relationship between  these two, 
which unites  them into a  canonical decomposition (Theorem~\ref{thm:base}). 
We have also presented an $O(nm)$ time algorithm for computing the basilica decomposition. 
As canonical decompositions have formed  the theoretical foundation of matching theory, 
we believe that the results in this paper 
will be beneficial  to this field, and, by extension, to the entire field of combinatorics. 
We have already obtained some important results using the ideas in this paper. 
\begin{itemize} 
\item 
The structure of {\em barriers}, which are a classically important notion that corresponds to the dual of maximum matchings, has been revealed~\cite{DBLP:conf/cocoa/Kita13, kita2012canonical}. 
\item 
A new proof of Lov\'asz's {\em cathedral theorem}, which is an inductive characterization of  {\em saturated graphs}, 
has been obtained~\cite{kita2014alternative}. 
\item 
A purely graph theoretic proof of the celebrated {\em tight cut lemma}, 
which has  contributed to  almost all the results about the perfect matching polytope 
since 1982, has been obtained~\cite{kita2015graph}.  
\end{itemize}

\begin{acnt}
An early version of this work was presented in the papers~\cite{kita2012partially, DBLP:conf/isaac/Kita12}.
 The author would like to express gratitude to Richard Hoshino and Yusuke Kobayashi for carefully reading an early version of this paper  and for giving useful comments on the writing and presentation.  
\end{acnt}

\bibliographystyle{../splncs03.bst}
\bibliography{../isaac2012_full.bib,../partiallyams.bib,../nanao_kita.bib,../nanaokita.bib,../algebraic.bib}

\begin{thebibliography}{10}
\providecommand{\url}[1]{\texttt{#1}}
\providecommand{\urlprefix}{URL }

\bibitem{cc2005}
Carvalho, M.H., Cheriyan, J.: An {$O(VE)$} algorithm for ear decompositions of
  matching-covered graphs. ACM Transactions on Algorithms  1(2),  324--337
  (2005)

\bibitem{cheriyan1997}
Cheriyan, J.: Randomized $\tilde{O}({M}(|{V}|))$ algorithms for problems in
  matching theory. SIAM Journal on Computing  26(6),  1635--1669 (1997)

\bibitem{DBLP:books/daglib/0030488}
Diestel, R.: Graph Theory, 4th Edition, Graduate texts in mathematics, vol.
  173. Springer (2012)

\bibitem{duff1986direct}
Duff, I.S., Erisman, A.M., Reid, J.K.: Direct methods for sparse matrices.
  Clarendon press Oxford (1986)

\bibitem{dm1958}
Dulmage, A.L., Mendelsohn, N.S.: Coverings of bipartite graphs. Canadian
  Journal of Mathematics  10,  517--534 (1958)

\bibitem{dm1959}
Dulmage, A.L., Mendelsohn, N.S.: A structure theory of bipartite graphs of
  finte exterior dimension. Transactions of the Royal Society of Canada,
  Section {III}  53,  1--13 (1959)

\bibitem{dm1963}
Dulmage, A.L., Mendelsohn, N.S.: Two algorithms for bipartite graphs. Journal
  of the Society for Industrial and Applied Mathematics  11(1),  183--194
  (1963)

\bibitem{edmonds1965}
Edmonds, J.: Paths, trees and flowers. Canadian Journal of Mathematics  17,
  449--467 (1965)

\bibitem{fujishige2005}
Fujishige, S.: Submodular Functions and Optimization. Elsevier Science, second
  edn. (2005)

\bibitem{gallai1964}
Gallai, T.: Maximale systeme unabh{\"{a}}ngiger kanten. A Magyer
  Tudom{\'{a}}nyos Akad{\'{e}}mia: 
  Int{\'{e}}zet{\'{e}}nek K{\"{o}}zlem{\'{e}}nyei  8,  401--413 (1964)

\bibitem{geelen2000}
Geelen, J.F.: An algebraic matching algorithm. Combinatorica  20(1),  61--70
  (2000)

\bibitem{grotschel2012geometric}
Gr{\"o}tschel, M., Lov{\'a}sz, L., Schrijver, A.: Geometric algorithms and
  combinatorial optimization, vol.~2. Springer Science \& Business Media (2012)

\bibitem{kita2012c}
Kita, N.: A generalization of the {D}ulmage-{M}endelsohn decomposition for
  arbitrary graphs, preprint

\bibitem{kita2012canonical}
Kita, N.: A canonical characterization of the family of barriers in general
  graphs. arXiv preprint arXiv:1212.5960  (2012)

\bibitem{kita2012partially}
Kita, N.: A partially ordered structure and a generalization of the canonical
  partition for general graphs with perfect matchings. arXiv preprint
  arXiv:1205.3816  (2012)

\bibitem{DBLP:conf/isaac/Kita12}
Kita, N.: A partially ordered structure and a generalization of the canonical
  partition for general graphs with perfect matchings. In: Chao, K., Hsu, T.,
  Lee, D. (eds.) Algorithms and Computation - 23rd International Symposium,
  {ISAAC} 2012, Taipei, Taiwan, December 19-21, 2012. Proceedings. Lecture
  Notes in Computer Science, vol. 7676, pp. 85--94. Springer (2012)

\bibitem{DBLP:conf/cocoa/Kita13}
Kita, N.: Disclosing barriers: {A} generalization of the canonical partition
  based on lov{\'{a}}sz's formulation. In: Widmayer, P., Xu, Y., Zhu, B. (eds.)
  Combinatorial Optimization and Applications - 7th International Conference,
  {COCOA} 2013, Chengdu, China, December 12-14, 2013, Proceedings. Lecture
  Notes in Computer Science, vol. 8287, pp. 402--413. Springer (2013)

\bibitem{kita2014alternative}
Kita, N.: An alternative proof of {L}ovasz's cathedral theorem. Journal of the
  Operations Research Society of Japan  57(1),  15--34 (2014)

\bibitem{kita2015graph}
Kita, N.: A graph theoretic proof of the tight cut lemma. arXiv preprint
  arXiv:1512.08870  (2015)

\bibitem{kotzig1959a}
Kotzig, A.: Z te\'orie kone\v{c}n\'ych grafov s line\'arnym faktorom. {I} ({\em
  in slovak}). Mathematica Slovaca  9(2),  73--91 (1959)

\bibitem{kotzig1959b}
Kotzig, A.: Z te\'orie kone\v{c}n\'ych grafov s line\'arnym faktorom. {II}
  ({\em in slovak}). Mathematica Slovaca  9(3),  136--159 (1959)

\bibitem{kotzig1960}
Kotzig, A.: Z te\'orie kone\v{c}n\'ych grafov s line\'arnym faktorom. {III}
  ({\em in slovak}). Mathematica Slovaca  10(4),  205--215 (1960)

\bibitem{lovasz1972a}
Lov\'asz, L.: A note on factor-critical graphs. Studia Scientiarum
  Mathematicarum Hungarica  7,  279--280

\bibitem{lovasz1972b}
Lov\'asz, L.: On the structure of factorizable graphs. Acta Mathematica
  Hungarica  23(1--2),  179--195 (1972)

\bibitem{lp1986}
Lov\'{a}sz, L., Plummer, M.D.: Matching Theory. Elsevier Science (1986)

\bibitem{lovasz1972structure}
Lov{\'{a}}sz, L.: {On the structure of factorizable graphs}. Acta Math.
  Hungarica  23(1-2),  179--195 (1972)

\bibitem{DBLP:conf/fct/Lovasz79}
Lov{\'{a}}sz, L.: On determinants, matchings, and random algorithms. In: {FCT}.
  pp. 565--574 (1979)

\bibitem{mv1980}
Micali, S., Vazirani, V.V.: An ${O}(\sqrt{|v|}\cdot{|E|})$ algorithm for
  finding maximum matching in general graphs. In: Proceedings of the 21st
  Annual IEEE Symposium on Foundations of Computer Science. pp. 17--27 (1980)

\bibitem{nakamura1988}
Nakamura, M.: Structural theorems for submodular functions, polymatroids and
  polymatroid intersections. Graphs and Combinatorics  4,  257--284 (1988)

\bibitem{schrijver2003}
Schrijver, A.: Combinatorial Optimization: Polyhedra and Efficiency.
  Springer-Verlag (2003)

\bibitem{DBLP:journals/combinatorica/SzegedyS06}
Szegedy, B., Szegedy, C.: Symplectic spaces and ear-decomposition of matroids.
  Combinatorica  26(3),  353--377 (2006)

\bibitem{vazirani1994}
Vazirani, V.V.: A theory of alternating paths and blossoms for proving
  correctness of the ${O}(\sqrt{V}{E})$ general graph maximum matching
  algorithm. Combinatorica  14,  71--109 (1994)

\end{thebibliography}

\end{document}